\documentclass[12pt,draftcls,onecolumn]{IEEEtran}

\usepackage{graphicx, amsmath,amssymb,color,booktabs}
\usepackage[colorlinks=true, letterpaper,    
      urlcolor=black,     
      citecolor=blue,
      menucolor=black,    
      linkcolor=black,    
      pagecolor=black,    
      breaklinks=true,
]{hyperref}

\usepackage{subfigure}

\usepackage{amsthm}

\newtheorem{assum}{Assumption}
\newtheorem{prop}{Proposition}
\newtheorem{thm}{Theorem}
\newtheorem{rem}{Remark}
\newtheorem{lem}{Lemma}

\newtheorem{ex}{Example}

\begin{document}

\ifpdf
\DeclareGraphicsExtensions{.pdf, .jpg, .tif}
\else
\DeclareGraphicsExtensions{.eps, .jpg}
\fi

\title{Parameter and state estimation of nonlinear systems using a multi-observer under the supervisory framework}
\author{Michelle S. Chong, Dragan Ne\v{s}i\'{c}, Romain Postoyan and Levin Kuhlmann 
\thanks{M. Chong, D. Ne\v{s}i\'{c} and L. Kuhlmann are with the Department of Electrical and Electronic Engineering, the University of Melbourne, Australia.
       \tt\small \{chongms,dnesic,levink\}@unimelb.edu.au}%
\thanks{R. Postoyan is with the Universit\'{e} de Lorraine, CRAN, UMR 7039 and the CNRS, CRAN, UMR 7039, France. He is financially supported by the ANR under the grant SEPICOT (ANR 12 JS03 004 01). 							\tt\small{romain.postoyan@univ-lorraine.fr}}%
}%
\maketitle \vspace{-4em}
\begin{abstract}
	We present a hybrid scheme for the parameter and state estimation of nonlinear continuous-time systems, which is inspired by the supervisory setup used for control. State observers are synthesized for some nominal parameter values and a criterion is designed to select one of these observers at any given time instant, which provides state and parameter estimates. Assuming that a persistency of excitation condition holds, the convergence of the parameter and state estimation errors to zero is ensured up to a margin, which can be made as small as desired by increasing the number of observers. To reduce the potential computational complexity of the scheme, we explain how the sampling of the parameter set can be dynamically updated using a zoom-in procedure. This strategy typically requires a fewer number of observers for a given estimation error margin compared to the static sampling policy. The results are shown to be applicable to linear systems and to a class of nonlinear systems. We illustrate the applicability of the approach by estimating the synaptic gains and the mean membrane potentials of a neural mass model.
\end{abstract}

\section{Introduction} \label{sec:intro}
The estimation of states and parameters is a long-standing problem in control theory, which is particularly involved when dealing with nonlinear systems.  While approaches are available to estimate the states (\emph{e.g.} \cite{andrieu2006existence, arcak2001nonlinear, besancon2007nonlinear, karagiannis2008invariant,khalil1999high}) or the parameters (\emph{e.g.} \cite{adetola2008finite}, \cite[Chapter 4]{ioannou1996robust}, \cite{ljung1999system}), the simultaneous estimation of parameter and state remains a challenging problem \cite{besancon2007nonlinear}. A classical technique consists of augmenting the state vector with the parameter vector and thus reducing the problem to state estimation only. However, this complicates the problem even for linear systems since rewriting the system in these new coordinates can turn the system highly nonlinear. We note that while nonlinear Kalman filters are often used in practice, their convergence is only provable under limited conditions \cite[Chapter 8]{anderson1979optimal}, \cite{julier2004unscented}, in addition, these filters may be difficult to tune.

In this paper, we investigate the problem of state and parameter estimation of deterministic nonlinear continuous-time systems by adopting an architecture known as the supervisory framework, see \cite[Chapter 6]{liberzon2003switching}. We assume that the system parameters are {constant} and that they belong to a known, compact set. We sample the parameter set to form a finite set of nominal values. A state observer is designed for each of these nominal values, which is assumed to satisfy a given robustness property with respect to parameter mismatches. They form the multi-observer unit \cite[Chapter 6]{liberzon2003switching}. The {supervisory} unit then provides state and parameter estimates by selecting one observer from the bank at any given time instant. We call this setup a \emph{supervisory observer}. The parameter and state estimates are guaranteed to converge to the true values up to any given margin, provided that the number of observers is sufficiently large. To overcome the possible computational burden of implementing a large number of observers, we propose a dynamic sampling policy of the parameter set, which is inspired by the quantization strategy used for control in \cite{liberzon2003hybrid}. The idea is to update the sampling of the parameter set iteratively via a zoom-in procedure. With this dynamic scheme, we show that the parameter and state estimation errors also converge to the origin up to an adjustable margin with the added benefit of typically requiring less  observers, as illustrated on an example.

The supervisory framework has traditionally been used for control, see \cite{battistelli2012supervisory,  hespanha2003hysteresis, hespanha2001multiple, morse1997supervisory, vu2011supervisory}. In these works, the system dynamics depend on some unknown parameters and the objective is to steer the system state to the origin, however no guarantee is provided on the parameter estimates. A similar approach to the one we propose is pursued for the estimation of linear systems in \cite{aguiar2007convergence, aguiar2008identification, han2012new}, \cite[Section 8.5]{anderson1979optimal}, \cite{li1996multiple}, where multiple observers are employed with different selection criteria. It has to be noted that the idea of dynamically sampling the parameter set has been used in \cite{han2012new} for linear systems under the assumption that the full state is measured, which is not the case in this study. In addition, we address nonlinear systems and we envision a different methodology compared to \cite{han2012new}.

We believe that this paper illustrates the potential of casting the problem of estimation of continuous-time systems in the hybrid systems setting. While hybrid tools have proved their efficiency for numerous control problems (see for example \cite{goebel2009hybrid, nesic2008stability, prieur2001uniting} to name a few), few works have investigated the estimation problem from this angle, see \cite{adetola2008finite}, \cite{hartman2012robust} (for parameter estimation) and \cite{karafyllis2011hybrid} (for state estimation). To the best of our knowledge, this paper is the first to address the joint estimation of states and parameters from this perspective. We believe that the advantages of supervisory control mentioned in \cite{hespanha2003overcoming} translate well to estimation. Firstly, it is not necessary to construct an adaptive observer, which is very challenging for nonlinear systems, see \emph{e.g.} \cite{besancon2007nonlinear, farza2009adaptive, zhang2002adaptive,tyukin2007adaptation}. Secondly, the supervisory framework has the advantage of modularity. Each of the components, namely the selection criterion, the monitoring signals and the multi-observer, can be designed independently to satisfy the respective properties required to meet our objective. This allows for the usage of readily available {state} observers for the additional purpose of parameter estimation. We thus show that Luenberger observers and a robust form of circle criterion observers \cite{arcak2001nonlinear, chong2012robust} can be used within this framework, noting that the approach is applicable to other types of state observers.

This work is motivated by the great need for estimation in neuroscience, especially for developing new methods for the classification or diagnosis of neurological diseases, such as epilepsy \cite{schiff2011neural}. To illustrate the applicability of the framework, we implement the proposed algorithm to estimate the mean membrane potentials (states) and the synaptic gains (parameters) of a neural mass model \cite{jansen1995eeg} which is able to realistically reproduce patterns seen in electroencephalographic (EEG) recordings.

The paper is organised as follows. In Section \ref{sec:prelim}, we introduce the mathematical notation used in the paper. The problem is stated in Section \ref{sec:prob} and we describe the supervisory setup in Section \ref{sec:framework_static} for the static sampling policy. The dynamic sampling policy is presented in Section \ref{sec:framework_dynamic}. The results obtained for both policies are shown to be applicable to linear systems and a class of nonlinear systems in Section \ref{sec:special}. The illustrative example from neuroscience is presented in Section \ref{sec:example}. Lastly, Section \ref{sec:conclusion} concludes the paper. All proofs are provided in the Appendix.

\section{Preliminaries} \label{sec:prelim}
Let $\mathbb{R}=(-\infty,\infty)$, $\mathbb{R}_{\geq 0}=[0,\infty)$,
$\mathbb{R}_{>0}=(0,\infty)$, $\mathbb{N}=\{0,1,2,\dots\}$ and $\mathbb{N}_{\geq 1}=\{1,2,\dots\}$.
 	The notation $(u,v)$ stands for $[u^{T} \; v^{T}]^{T}$, where $u\in\mathbb{R}^{m}$ and $v\in \mathbb{R}^{n}$.
	For a vector $x\in\mathbb{R}^{n}$,  $|x|$ denotes the Euclidean norm of $x$ and $|x|_{\infty}$ denotes the infinity norm of $x$, i.e.  $|x|_{\infty}:=\max\{|x_1|,\dots,|x_n|\}$ where $x=(x_1,\dots,x_n)$.
	Let $\textrm{diag}(a_1,\dots,a_n)$ stand for the diagonal matrix with real elements $a_1,\dots,a_n$.
	The maximum (minimum) eigenvalue of a real, symmetric matrix $P$ is denoted $\lambda_{\max}(P)$ ($\lambda_{\min}(P)$).
	The symmetric block component of a symmetric matrix is denoted by $\star$.
	The notation $\mathbb{I}$ stands for the identity matrix.
	The hypercube centered at $\xi \in \mathbb{R}^{n}$ of edge length $2r>0$ is denoted by $\mathcal{H}(\xi,{r}) := \{x\in\mathbb{R}^{n} : \left|x - \xi \right|_{\infty} \leq r \}$.
	For any $\Delta>0$, let the set of piecewise continuous functions from $\mathbb{R}_{\geq 0}$ to $\mathcal{H}(0,\Delta)$ be  $\mathcal{M}_{\Delta}$.
	The left-limit operator is denoted by $(\cdot)^{-}$.
 	A continuous function $\alpha:\mathbb{R}_{\geq 0} \rightarrow \mathbb{R}_{\geq 0}$ is a class $\mathcal{K}$ function, if it is strictly increasing and $\alpha(0)=0$; additionally, if $\alpha(r)\to\infty$ as $r \to \infty$, $\alpha$ is a class $\mathcal{K}_{\infty}$ function.
	A continuous function $\beta:\mathbb{R}_{\geq 0}\times \mathbb{R}_{\geq 0} \rightarrow \mathbb{R}_{\geq 0}$ is a class $\mathcal{KL}$ function, if $\beta(\cdot,s)$ is a class $\mathcal{K}$ function for each $s\geq 0$ and $\beta(r,.)$ is non-increasing and $\beta(r,s)\rightarrow 0$ as $s\rightarrow \infty$ for each $r\geq 0$.

\section{Problem formulation} \label{sec:prob}
Consider the system
\begin{eqnarray}
	\dot{x} & = & f(x,p^{\star},u) \nonumber \\
	y & = & h(x,p^{\star}), \label{eq:general_plant}
\end{eqnarray}
where the state is $x\in\mathbb{R}^{n_x}$, the measured output is $y\in\mathbb{R}^{n_y}$, the input is $u \in \mathbb{R}^{n_u}$ which is known and the \emph{unknown} parameter vector $p^{\star}\in \Theta \subset \mathbb{R}^{n_p}$ is constant, where $\Theta$ is a {known} compact set. For any initial condition and any piecewise-continuous input $u$, system \eqref{eq:general_plant} admits a unique solution that is defined for all positive time. We make the following assumption on system (\ref{eq:general_plant}).

\begin{assum} \label{ass:stable_sys}
	The solutions to system \eqref{eq:general_plant} are uniformly bounded, i.e. for all $\Delta_{x}$, $\Delta_u \geq 0$, there exists a constant $K_{x}=K_{x}(\Delta_x,\Delta_u) > 0$ such that for all $x(0)\in \mathcal{H}(0,\Delta_{x})$ and $u\in \mathcal{M}_{\Delta_u}$
	\begin{equation}
	 	|x(t)|_{\infty} \leq K_x, \qquad \forall t\geq 0. \label{eq:stable_sys_bound}
	\end{equation}\hfill $\Box$
\end{assum}
It has to be noted that the bound $K_x$ in \eqref{eq:stable_sys_bound} does not need to be known to implement the estimation algorithms presented in Sections \ref{sec:framework_static} and \ref{sec:framework_dynamic}, only its existence has to be ensured. Furthermore, Assumption \ref{ass:stable_sys} can be relaxed in some cases as explained later in Remark \ref{rem:obs}.
\begin{rem}
	Contrary to the problem of supervisory control in \cite{vu2011supervisory}, we do not require system \eqref{eq:general_plant} to be stabilisable. Since our purpose is to estimate, we only require the solutions to system \eqref{eq:general_plant} to be uniformly bounded. This is a reasonable assumption in the estimation context since most physical systems have solutions that are uniformly bounded, as illustrated by the neural mass models investigated in Section \ref{sec:example}.
\end{rem}

The main objective of this paper is to estimate the parameter vector $p^{\star}$ and the state $x$ of system \eqref{eq:general_plant} when only the input $u$ and the output $y$ of system \eqref{eq:general_plant} are measured, using a supervisory observer.

\section{Supervisory observer with a static sampling policy} \label{sec:framework_static}
\subsection{Description}
Inspired by the supervisory framework used for control \cite[Chapter 6]{liberzon2003switching}, the proposed methodology consists of two basic units (see Figure \ref{fig:supervisory_setup}): a bank of state observers (multi-observer) which generates state estimates and the supervisor (monitoring signals and a selection criterion) which chooses one observer at any given time. The estimated parameters and states are derived from the choice the supervisor makes.
\begin{figure}[h!]
	\begin{center}
 		\includegraphics[height=7cm,width=8.9cm]{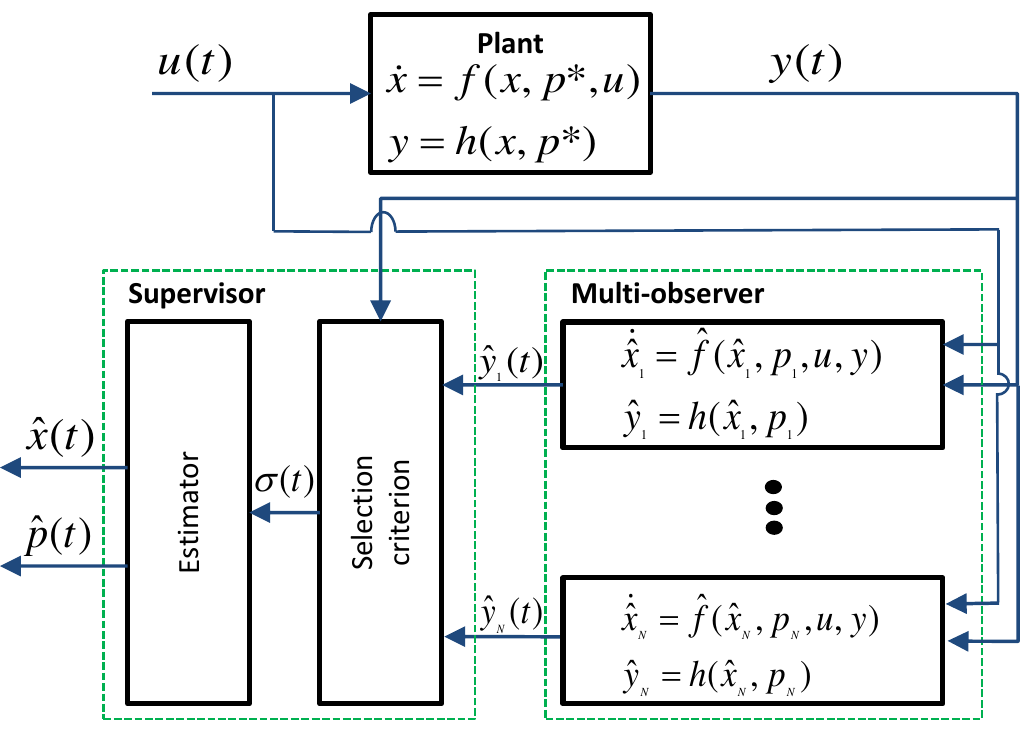}
	\end{center} \vspace{-2em}
	\caption{Supervisory observer. \label{fig:supervisory_setup}}		
\end{figure}
\vspace{1em}
\subsubsection{Static sampling of the parameter set $\Theta$} \label{sec:static_sampling}
We select $N\in\mathbb{N}_{\geq 1}$ parameter values $p_i$, $i\in\{1,\dots,N\}$, in the set $\Theta$ to form the sampled parameter set $\widehat{\Theta}:=\{p_1,\dots,p_N\}$. The selection of the samples is done in such a way that the distance from $p^{\star}$ to $\widehat{\Theta}$ defined as
\begin{equation}
	d(p^{\star},\widehat{\Theta}) := \underset{p\in\widehat{\Theta}}{\min} \left| p^{\star} - p \right|_{\infty}, \label{eq:pt_set_dist}
\end{equation}
tends to zero as $N$ tends to infinity. This can be guaranteed by embedding the parameter set $\Theta$ in a hyperrectangle (which is always possible because $\Theta$ is compact) and by employing uniform sampling, for instance. It may also be achieved using a logarithmic sampling if prior information is known, such as the probability distribution of the system parameter in $\Theta$. Figure \ref{fig:partition} illustrates these forms of sampling.

\begin{figure}[h!]
	\begin{center}
 		\subfigure[Uniform]{
			\includegraphics[scale=0.4]{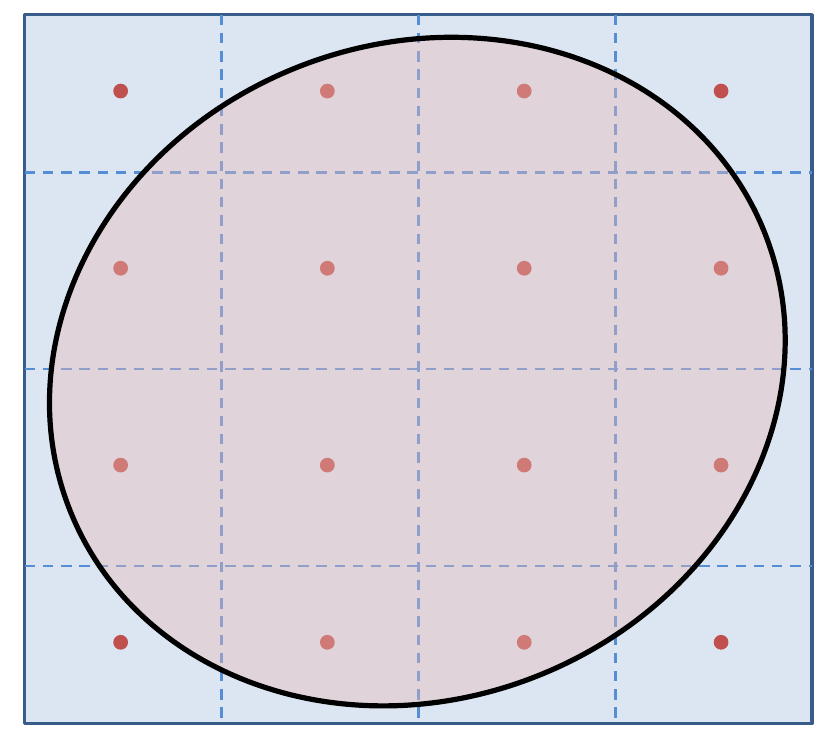}
		}
		\subfigure[Logarithmic]{
			\includegraphics[scale=0.4]{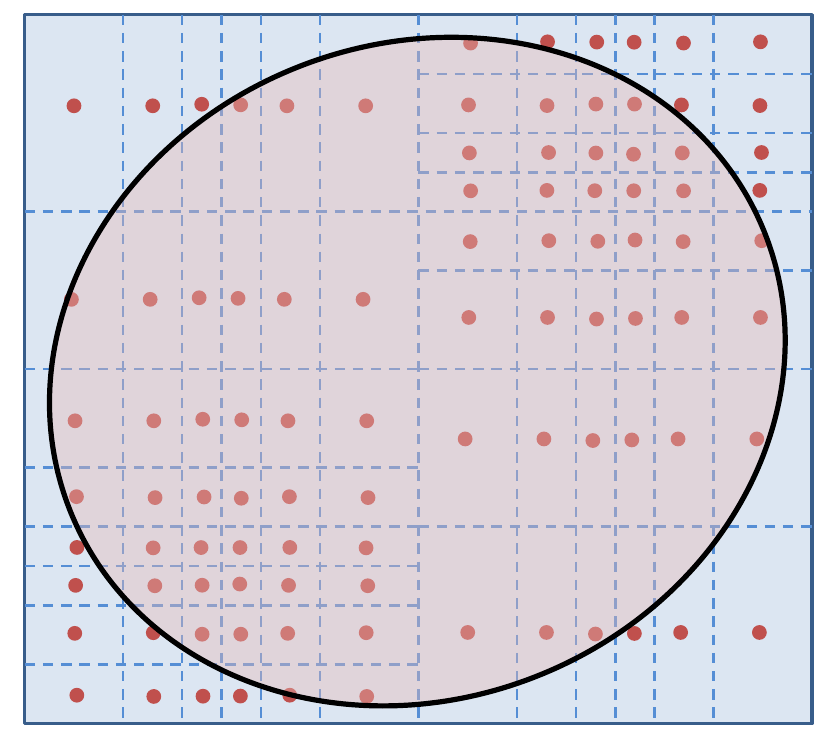}
		}
	\end{center} \vspace{-2em}
	\caption{Examples for the sampling of the parameter set $\Theta \subset \mathbb{R}^{2}$. Dots represent the parameters $p_{i}\in \widehat{\Theta}$, $i\in\{1,\ldots,N\}$. The shaded region in red indicates the parameter set $\Theta$ that has been embedded in a hyperrectangle (shaded region in blue).\label{fig:partition}}
\end{figure}

\vspace{1em}

\subsubsection{Multi-observer} \label{sec:multiobs}
A state observer is designed for each $p_{i}$, $i\in\{1,\dots,N\}$,
\begin{eqnarray}
	\dot{\hat{x}}_{i} & = & \hat{f}(\hat{x}_{i},p_{i},u,y) \nonumber \\
	\hat{y}_{i} & = & h(\hat{x}_{i},p_{i}),  \label{eq:observer_i}
\end{eqnarray}
where $\hat{x}_{i}\in\mathbb{R}^{n_x}$ is the state estimate and $\hat{y}_{i}\in\mathbb{R}^{n_y}$ is the output estimate. The solutions to \eqref{eq:observer_i} are assumed to be unique and defined for all positive time for all initial conditions, any piecewise-continuous input $u$, any system output $y$ and any parameter $p_i \in {\Theta}$.  Denoting the state estimation error as $\tilde{x}_{i}:=\hat{x}_{i}-x$, the output error as $\tilde{y}_{i}:=\hat{y}_{i}-y$ and the parameter error as $\tilde{p}_{i}:=p_{i}-p^{\star}$, we obtain the following state estimation error systems, for $i\in\{1,\dots,N\}$,
\begin{equation}
	\begin{array}{lll}
	\dot{\tilde{x}}_{i} & = & \hat{f}(\tilde{x}_{i}+x,\tilde{p}_{i}+p^{\star},u,y) - f(x,p^{\star},u)
	 =:  F_{i}(\tilde{x}_{i}, x, \tilde{p}_{i}, p^{\star}, u)  \\
	\tilde{y}_{i} & = & h(\tilde{x}_{i}+x,\tilde{p}_{i}+p^{\star}) - h(x,p^{\star}) =: H(\tilde{x}_{i}, x, \tilde{p}_{i}, p^{\star}).
	\end{array} \label{eq:error_sys_general}
\end{equation}
We assume that the observers \eqref{eq:observer_i} are designed such that the following property holds.

\begin{assum}\label{ass:obs_error_i}
	There exist scalars $a_{1}$, $a_{2}$, $\lambda_{0}>0$ and a continuous non-negative function $\tilde{\gamma}:\mathbb{R}^{n_p}\times \mathbb{R}^{n_x} \times \mathbb{R}^{n_u} \rightarrow \mathbb{R}_{\geq 0}$ with $\tilde{\gamma}(0,x,u)=0$ for all $x \in \mathbb{R}^{n_x}$, $u \in \mathbb{R}^{n_u}$ such that for any $p_i \in \Theta$, $i\in\{1,\dots,N\}$, there exists a continuously differentiable function $V_{i}: \mathbb{R}^{n_x} \rightarrow \mathbb{R}_{\geq 0}$, which satisfies the following for all $\tilde{x}_{i} \in\mathbb{R}^{n_x}$, $x\in\mathbb{R}^{n_x}$, $u\in\mathbb{R}^{n_u}$,
	\begin{equation}
		a_{1}|\tilde{x}_{i}|_{\infty}^{2} \leq V_{i}(\tilde{x}_{i}) \leq a_{2}|\tilde{x}_{i}|_{\infty}^{2}, \label{eq:lyap_iss_1}
	\end{equation}
	\begin{equation}
			\frac{\partial V_i}{\partial \tilde{x}_{i}} F_{i}(\tilde{x}_{i},\tilde{p}_{i},p^{\star},u,x) \leq -\lambda_{0} V_{i}(\tilde{x}_{i}) + \tilde{\gamma}(\tilde{p}_{i},x,u). \label{eq:lyap_iss_2}
	\end{equation}	 \hfill $\Box$
\end{assum}
Assumption \ref{ass:obs_error_i} means that we know how to design a state observer for system \eqref{eq:general_plant} which is robust to parameter errors. When $\tilde{p}_{i}=0$, i.e. when $p^{\star}$ is known, inequalities \eqref{eq:lyap_iss_1} and \eqref{eq:lyap_iss_2} imply that the origin of the state estimation error system \eqref{eq:error_sys_general} is globally exponentially stable (note that (\ref{eq:lyap_iss_1}) can be equivalently stated with Euclidean norms as all norms are equivalent in $\mathbb{R}^{n_x}$). When $\tilde{p}_{i}\neq 0$, condition \eqref{eq:lyap_iss_2} needs to be satisfied, which is the case when system \eqref{eq:error_sys_general} is input-to-state {exponentially} stable  \cite{grune1999asymptotic} with respect to $\tilde{p}_{i}$ for example. In Section \ref{sec:special}, we will show that Luenberger observers and a class of nonlinear observers based on the circle criterion \cite{arcak2001nonlinear, chong2012robust} satisfy Assumption \ref{ass:obs_error_i}.

Following the terminology used in \cite{vu2011supervisory}, we call the bank of observers in \eqref{eq:observer_i} the multi-observer.

\begin{rem} \label{rem:obs}
	When Assumption \ref{ass:obs_error_i} holds with $\tilde{\gamma}$ which only depends on $\tilde{p}_{i}$ and $u$, we do not need system \eqref{eq:general_plant} to satisfy Assumption \ref{ass:obs_error_i} to guarantee the convergence of the estimates.
\end{rem}
\begin{rem}
	The Lyapunov-based conditions \eqref{eq:lyap_iss_1}-\eqref{eq:lyap_iss_2} stated in Assumption \ref{ass:obs_error_i} differ from the conditions in \cite[Equations (10a) and (10b)]{vu2011supervisory} and \cite[Theorem 4.3 (iii)]{vu2007iss} because the objective is to estimate the parameters and the states of \eqref{eq:general_plant}, whereas the available results consider the problem of stabilization of an equilibrium point without guarantees on the convergence of the parameter estimates.
\end{rem}
\vspace{1em}
\subsubsection{Monitoring signal} \label{sec:monitoring_signal}
Similar to \cite[Equation 6]{vu2011supervisory}, the monitoring signal associated with each observer is the exponentially weighted $\mathcal{L}_{\infty}$ norm \cite{ioannou1996robust} of the output error defined as, for $i\in\{1,\dots,N\}$,
\begin{equation}
	\mu_{i}(t) =  \int_{0}^{t} \exp(-\lambda (t-s)) |\tilde{y}_{i}(s)|_{\infty}^{2} ds,\qquad \forall t\geq 0, \label{eq:monitoring_signal}
\end{equation}
where $\lambda>0$ is a design parameter. The monitoring signal \eqref{eq:monitoring_signal} can be implemented as a linear filter, for $i \in \{1,\dots,N\}$,
\begin{equation}
\begin{array}{rllll}
	\dot{\mu}_{i}(t) & = & -\lambda \mu_{i}(t) + |\tilde{y}_{i}(t)|_{\infty}^{2} \qquad \forall t\geq 0\\
\mu_i(0) & = & 0.
\label{eq:cost_filter}
\end{array}
\end{equation}
We assume that the output error of each of the observers $\tilde{y}_{i}$ satisfies the following property.
\begin{assum} \label{ass:PE_y_tilde}
	For any $\Delta_{\tilde{x}}$, $\Delta_{x}$, $\Delta_{u}>0$, there exist a class $\mathcal{K}_{\infty}$ function $\alpha_{\tilde{y}}$ and a constant $T_{f}=T_{f}(\Delta_{\tilde{x}}, \Delta_{x}, \Delta_{u})>0$ such that for all $p_{i}\in\Theta$, with $i\in\{1,\ldots,N\}$, $\tilde{x}_{i}(0)\in \mathcal{H}(0,{\Delta_{\tilde{x}}})$, $x(0) \in \mathcal{H}(0,{\Delta_{{x}}})$, for some $u\in \mathcal{M}_{\Delta_u}$, the corresponding solution to systems \eqref{eq:general_plant} and \eqref{eq:error_sys_general} satisfies
	\begin{equation}
		\int_{t-T_{f}}^{t} |\tilde{y}_{i}(\tau)|_{\infty}^{2} d\tau \geq \alpha_{\tilde{y}}(|\tilde{p}_{i}|_{\infty}), \qquad \forall t\geq  T_f. \label{eq:tilde_y_PE}
	\end{equation} \hfill $\Box$
\end{assum}
The inequality \eqref{eq:tilde_y_PE} is known as a \emph{persistency of excitation (PE) condition} that appears in the identification and adaptive literature \cite{sastry2011adaptive}. It differs from the classical PE definition \cite[Definition 2.5.3]{sastry2011adaptive} in that we take the norm of the signal and we consider a family of systems \eqref{eq:error_sys_general} parameterised by $\tilde{p}_{i}$, where we consider only the lower bound (excitation level) in \eqref{eq:tilde_y_PE} which depends on $\tilde{p}_{i}$. In particular, if $\tilde{p}_{i}=0$, we do not require any PE property. In addition, the excitation level grows with the norm of the parameter error $\tilde{p}_{i}$. Hence, the integral term in \eqref{eq:tilde_y_PE} provides quantitative information about the parameter estimation error. Assumption \ref{ass:PE_y_tilde} holds when the output errors $\tilde{y}_{i}$,  $i\in\{1,\dots,N\}$, satisfy the classical PE condition according to the proposition below.
	\begin{prop} \label{prop:classic_PE}
		Consider system \eqref{eq:general_plant} and the observer \eqref{eq:observer_i}. Suppose that the following holds.
		\begin{enumerate}
			\item Assumption \ref{ass:stable_sys} is satisfied.
			\item The functions $f$ and $h$ in \eqref{eq:general_plant} and $\hat{f}$ in \eqref{eq:observer_i} are continuously differentiable.
			\item For all $\Delta_{\tilde{x}}$, $\Delta_{x}$, $\Delta_{u}>0$ and any $p_{i}\in\Theta$, $i\in\{1,\ldots,N\}$, with $\tilde{p}_{i}\neq 0$ there exist constants $T_{f}=T_{f}(\Delta_{\tilde{x}}, \Delta_{x}, \Delta_{u})$ and $\bar{\alpha}_{i}=\bar{\alpha}_{i}(\Delta_{\tilde{x}}, \Delta_{x}, \Delta_{u}, \tilde{p}_{i})>0$ such that for all
$\tilde{x}_{i}(0)\in \mathcal{H}(0,\Delta_{\tilde{x}})$, $x(0) \in \mathcal{H}(0,{\Delta_{x}})$ and for some $u\in \mathcal{M}_{\Delta_u}$, the corresponding solution to \eqref{eq:general_plant}, \eqref{eq:error_sys_general} satisfies 
			\begin{equation}
				\begin{array}{ll}
					\int_{t-T_{f}}^{t} \tilde{y}_{i}(\tau) \tilde{y}_{i}(\tau)^{T} d\tau \geq \bar{\alpha}_{i} \mathbb{I}, & \forall t\geq T_f. \\
				\end{array}
				 \label{eq:tilde_y_PE_classic}
			\end{equation}
		\end{enumerate}
		Then \eqref{eq:tilde_y_PE} holds.		\hfill $\Box$
	\end{prop}
Proposition \ref{prop:classic_PE} indicates that tools to verify the classical PE condition \eqref{eq:tilde_y_PE_classic} can be applied to ensure the satisfaction of Assumption \ref{ass:PE_y_tilde}. The reader is referred to \cite[Chapter 6]{narendra1989stable} and \cite{panteley2001relaxed} where a priori checkable conditions are respectively given for linear and nonlinear systems.
	\vspace{1em}
\subsubsection{Selection criterion}
The signal $\sigma:\mathbb{R}_{\geq 0} \rightarrow \{1,\dots,N\}$ is used to choose an observer from the bank of $N$ observers at every instant of time. It is defined as
\begin{equation} \label{eq:selection_criterion}
	\sigma(t)  :=  \underset{i\in\{1,\dots,N\}}{\arg\min} \mu_{i}(t),\qquad \forall t \geq 0.
\end{equation}
Note that no dwell time is guaranteed with this selection criterion, therefore rapid changes of the signal $\sigma$ are allowed\footnote{The hysteresis based switching in \cite{hespanha2001multiple} can be used instead of \eqref{eq:selection_criterion} to ensure the existence of a finite number of discontinuities of $\sigma$ over any given finite-time interval. The results of Sections \ref{sec:gen_results_static} and \ref{sec:results-dynamic} still apply in this case provided that a positive constant is added to the monitoring signal as done in \cite[Equation 9]{hespanha2001multiple} and that the hysteresis constant is sufficiently small.}. Nevertheless, these switches do not affect the dynamics of the observers \eqref{eq:observer_i} and system \eqref{eq:general_plant}.
\vspace{1em}
\subsubsection{Parameter and state estimates}
Based on the signal $\sigma$ in \eqref{eq:selection_criterion}, the estimated parameters and states respectively are, for all $t\geq 0$,
\begin{eqnarray}
	\hat{p}(t) & := & p_{\sigma(t)}, \label{eq:estimated_parameter} \\
	\hat{x}(t) & := & \hat{x}_{\sigma(t)}(t). \label{eq:estimated_state}
\end{eqnarray}
The parameter and the state estimates are discontinuous in general because these signals switch among a finite family of continuous trajectories that are in general different at the switching instant.

\subsection{Convergence guarantees} \label{sec:gen_results_static}
The theorem below states that the estimated parameter $\hat{p}$ and state $\hat{x}$ in \eqref{eq:estimated_parameter} and \eqref{eq:estimated_state} are respectively guaranteed to converge to their true values $p^{\star}$ and $x$ up to some selected margins $\nu_{\tilde{p}}$ and $\nu_{\tilde{x}}>0$, provided that the number of observers $N$ is sufficiently large.

\begin{thm} \label{thm:main_static}
	Consider system \eqref{eq:general_plant}, the multi-observer \eqref{eq:observer_i}, the monitoring signals \eqref{eq:monitoring_signal}, the selection criterion \eqref{eq:selection_criterion}, the parameter estimate \eqref{eq:estimated_parameter} and the state estimate \eqref{eq:estimated_state}. Suppose Assumptions  \ref{ass:stable_sys}-\ref{ass:PE_y_tilde} are satisfied. For any $\Delta_{\tilde{x}}$, $\Delta_x$, $\Delta_u>0$ and any margins $\nu_{\tilde{x}}$, $\nu_{\tilde{p}}>0$, there exist $T$, $\bar{K}_{\tilde{x}}>0$ and a sufficiently large $N^{\star}\in\mathbb{N}_{\geq 1}$ such that for any $N\geq N^{\star}$, the following holds for all $(x(0),\tilde{x}_{i}(0)) \in \mathcal{H}(0,{\Delta_x}) \times \mathcal{H}(0,{\Delta_{\tilde{x}}})$ for $i\in\{1,\dots,N\}$ and for any $u \in \mathcal{M}_{\Delta_u}$ that satisfies Assumption \ref{ass:PE_y_tilde},
	\begin{equation}
		\begin{array}{c}
		 \left|\tilde{p}_{\sigma(t)}(t) \right|_{\infty} \leq \nu_{\tilde{p}}\;\; \forall t\geq T,  \qquad \underset{t\to \infty}{\limsup}\left|\tilde{x}_{\sigma(t)}(t) \right|_{\infty} \leq \nu_{\tilde{x}}, \qquad \left| \tilde{x}_{\sigma(t)}(t) \right|_{\infty} \leq \bar{K}_{\tilde{x}}\;\;\forall t \geq 0. \label{eq:main_static_results}
		\end{array}
	\end{equation}
	\hfill $\Box$
\end{thm}

The accuracy of the estimates can be rendered as accurate as desired by increasing the number of observers $N$. To ensure the properties in (\ref{eq:main_static_results}), what we actually need is that a selected parameter $p_i$ is sufficiently close to $p^{\star}$ so that $d(p^{\star},\widehat{\Theta})$ in (\ref{eq:pt_set_dist}) is sufficiently small.  This is guaranteed by taking $N$ to be large with a sampling of $\Theta$ which ensures (\ref{eq:pt_set_dist}). For some applications, it is sufficient to have estimates that are accurate within some margin of error. This is the case for instance for the anticipation of abnormal neural behaviour such as seizures caused by epilepsy, where the model parameter set is composed of seizure and non-seizure related subsets \cite{wendling2005interictal}. Therefore, an algorithm that provides estimates of the parameters within some adjustable margin can be used to infer the region of the parameter set in which the true parameter lies.

A potential drawback of the scheme presented in this section is the need for a sufficiently large number of observers to ensure that the estimates fall within the required margins. While this may be feasible for some applications, it may be computationally intensive for others. We explain how to overcome this issue in the next section.

\section{Supervisory observer with a dynamic sampling policy} \label{sec:framework_dynamic}
Contrary to Section \ref{sec:framework_static}, we now update the sampling of the parameter set $\Theta$ at some time $t_k$, $k\in\mathbb{N}$. The proposed dynamic sampling policy builds upon the results of Section \ref{sec:framework_static} in that with a sufficient number of observers $N$, the parameter estimation error converges to a given margin in finite-time according to Theorem \ref{thm:main_static}. Once this happens, a new set of $N$ parameters is chosen within the hypercube centered at the latest parameter estimate with an edge length which is proportional to the aforementioned margin. This  aims to provide a better estimate of the parameter and the state for a given number of observers by zooming into the parameter set at each update time.

\subsection{Description}\label{sec:description-dynamic}
For the sake of simplicity and without loss in generality, we assume that $\Theta$ is a hypercube\footnote{We can always find a hypercube which contains $\Theta$ (since $\Theta$ is a compact set) and take the hypercube to be the parameter set. It has to be noted that to embed a given parameter set into a hypercube may be conservative, as we may end up working with a `much larger set' compared to the initial one. In this case, the parameters may be scaled to reduce this conservatism.} centered at some known $p_c\in\mathbb{R}^{n_p}$ and of edge length $2\Delta>0$. Let $N\in\mathbb{N}_{\geq 1}$ be the number of parameter samples and $\nu_{\tilde{p}}>0$ be the acceptable margin of error for the parameter estimate.

At the initial time $t_0=0$, we select $N$ values in $\Theta$ to form the initial sampled parameter set $\widehat{\Theta}(0):=\{p_1(0),\dots,p_{N}(0)\}$. The way the set $\widehat{\Theta}(0)$ is defined is as explained in the following. From the results of the static policy in Section \ref{sec:framework_static} (in particular \eqref{eq:main_static_results}), we know that after a sufficiently long time $T$ (see \eqref{eq:main_static_results}) and for a sufficiently large $N$, the norm of the parameter estimation error will be less than $C:=\max\{\alpha\Delta,\nu_{\tilde{p}}\}$ where $\alpha\in(0,1)$ is a design constant, which we call the zooming factor. In the case where $C=\alpha \Delta$, we know that $p^{\star}$ is in the hypercube $\mathcal{H}(\hat{p}(t_1^{-}),\alpha\Delta)$ for $t_{1} \geq T$. We can thus select $N$ parameter values in the hyperrectangle $\Theta(1):=\mathcal{H}(\hat{p}(t_1^{-}),\alpha\Delta) \cap \Theta$ (noting that $\mathcal{H}(\hat{p}(t_1^{-}),\alpha\Delta)$  is not necessarily included in $\Theta$) to form a new set of sampled parameters $\widehat{\Theta}(1):=\{p_{1}(1),\dots,p_{N}(1)\}$, see Figure \ref{fig:set}.
When $C=\nu_{\tilde{p}}$, the desired result is obtained. Nevertheless, we cannot detect this case on-line as we do not know $\tilde{p}_{\sigma}$. Hence, we apply the same procedure as when $C=\alpha\Delta$, even though $p^{\star}$ is no longer guaranteed to be in $\mathcal{H}(\hat{p}(t_1^{-}),\alpha\Delta)$, which is fine as the norm of the induced parameter estimation error will remain less than $\nu_{\tilde{p}}$ for all future times. This procedure is carried out iteratively at each update time instant $t_k$, $k\in\mathbb{N}$ which verifies
\begin{equation} \label{eq:update_time_interval}
	T_{d}:=t_{k+1}-t_{k}>0, \qquad \forall k\in\mathbb{N},
\end{equation}
where $T_d>0$ is a design parameter which is selected larger than $T$. 
The sets $\widehat{\Theta}(k)$, for $k\in\mathbb{N}$, are defined such that the property below is verified
\begin{equation}
\begin{array}{llllll}
p^{\star}\in\Theta(k)  & \Rightarrow & d\left(p^{*},\widehat{\Theta}(k)\right) \leq \pi(\Delta,N),
\end{array}\label{eq:dyn-sampling}
\end{equation}
where $\pi\in\mathcal{KL}$. In that way, for any $k\in\mathbb{N}$, $d(p^{*},\widehat{\Theta}(k))\to 0$ as $N\to\infty$ like in Section \ref{sec:framework_static} (as long as $p^{\star}\in\Theta(k)$). An example of sampling which ensures (\ref{eq:dyn-sampling}) is provided below.

\begin{ex}\label{ex:sampling} It can be noted that the sets $\Theta(k)$, $k\in\mathbb{N}$, constructed in this section are hyperrectangles, i.e. 
\begin{equation}
	\Theta(k)  =  \left\{p=(p^{1},\ldots,p^{n_{p}})\,:\,\forall j\in\{1,\ldots,n_{p}\}\,\,|p^{j}-p_{c}^{j}(k)|\leq\delta^{j}(k)\right\},
\end{equation}
where $p_{c}(k)=(p_{c}^{1}(k),\ldots,p_{c}^{n_{p}}(k))$ is the center and the constants $\delta^{j}(k)>0$, $j\in\{1,\ldots,n_{p}\}$, define the edge lengths. Moreover it holds that $\Theta(k)\subset\Theta$, hence $\delta^{j}(k)\leq\Delta$ for any $j\in\{1,\ldots,n_{p}\}$. Let $N=m^{n_{p}}$ with $m\in\mathbb{N}_{\geq 1}$, we partition $\Theta(k)$ into $N$ hyperrectangles such that their union is equal to $\Theta(k)$ and the intersection of their interiors is empty. These $N$ sets have the form $\{p=(p^{1},\ldots,p^{n_{p}})\,:\, \forall j\in\{1,\ldots,n_{p}\}\,\,|p^{j}-p_{c,i}^{j}(k)|\leq\frac{\delta^{j}(k)}{m}\}$ where $p_{c,i}$ is the center, $i\in\{1,\ldots,N\}$. Suppose $p^{\star}\in\Theta(k)$, then $p^{\star}=(p^{\star}_{1},\dots,p^{\star}_{n_p})$ belongs to (at least) one of these $N$ hyperrectangles. As a consequence, $|p^{\star}_{j}-p_{c,i^{*}}^{j}(k)|\leq\frac{\delta^{j}(k)}{m}$ for some $i^{\star}\in\{1,\ldots,N\}$ and any $j\in\{1,\ldots,n_{p}\}$. Consequently $|p^{\star}-p_{c,i^{\star}}|_{\infty} \leq {\underset{j\in\{1,\ldots,n_{p}\}}\max}\frac{\delta^{j}(k)}{m} \leq \frac{\Delta}{m}$ as $\delta^{j}(k)\leq\Delta$ for any $j\in\{1,\ldots,n_{p}\}$. Finally, we derive that $d\left(p^{\star},\widehat{\Theta}(k)\right)\leq|p^{\star}-p_{c,i^{\star}}|_{\infty}\leq \frac{\Delta}{N^{\frac{1}{n_{p}}}}$. Noting that $\pi:(s,r) \mapsto \min\left\{s,\frac{s}{r^{\frac{1}{n_p}}}\right\}$, we have that (\ref{eq:dyn-sampling}) holds. \hfill $\Box$
\end{ex}

\begin{figure}[h!]
	\includegraphics[width=8cm]{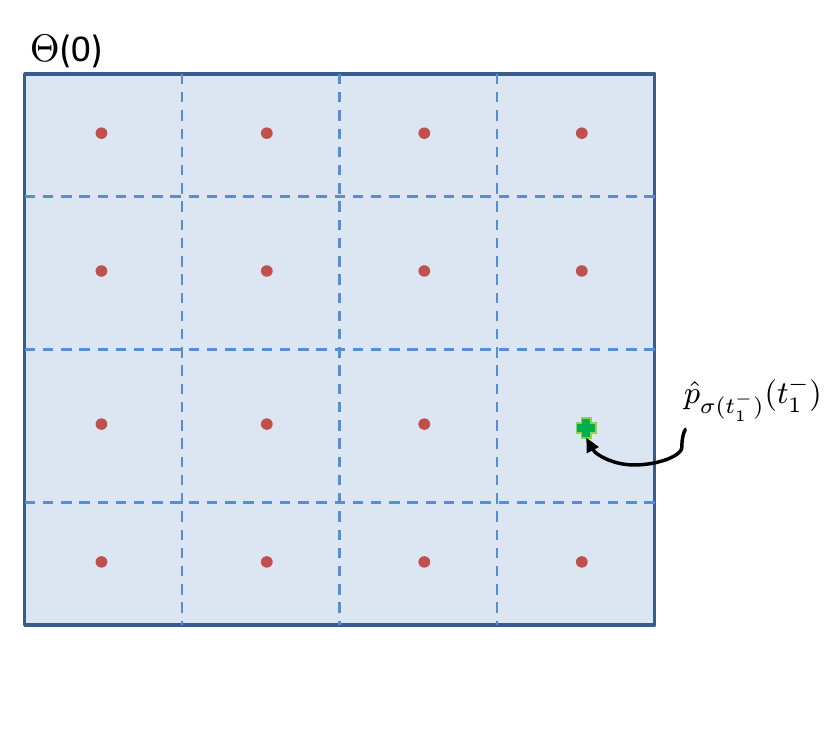}
	\includegraphics[width=8cm]{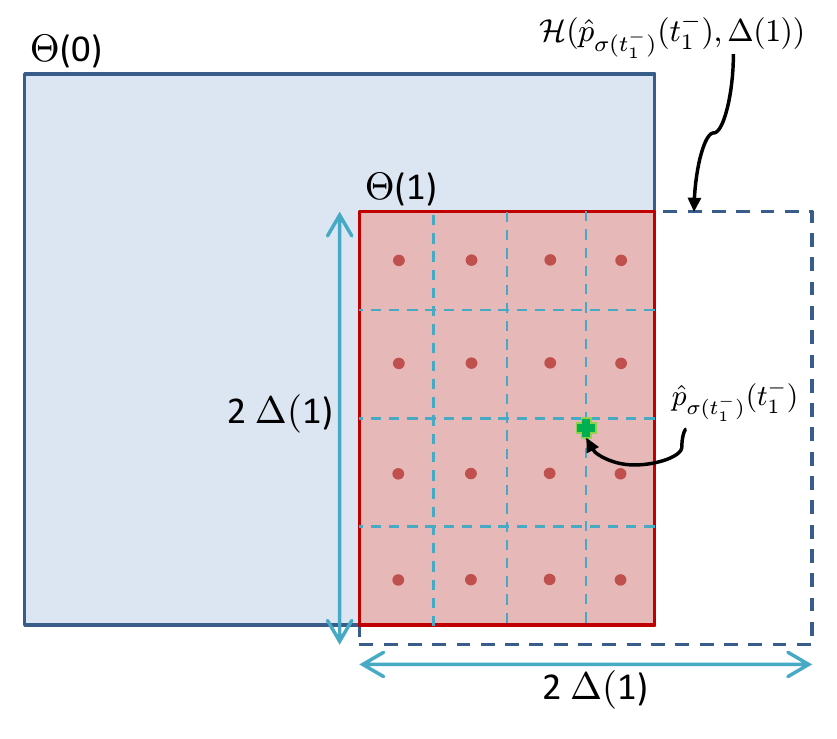}
	\vspace{-2em}
	\caption{Illustration of sets $\Theta(0)$ (shaded blue region), $\Theta(1)$ (shaded red region), $\widehat{\Theta}(0)$ and $\widehat{\Theta}(1)$. Legend: red dots indicate the sampled parameters $p_i$ for $i\in\{1,\dots,N\}$, green cross indicates the selected parameter $\hat{p}_{\sigma(t_1^{-})}(t_1^{-})$.} \label{fig:set}
\end{figure}

We summarize the dynamic sampling policy below.
\begin{itemize}
	\item \underline{At $t_0=0$.} Let $\Delta(0)=\Delta$ and $\Theta(0)=\Theta$. The set $\widehat{\Theta}(0)$ is obtained by discretizing the set $\Theta(0)$ with $N$ points such that (\ref{eq:dyn-sampling}) holds.

	\item \underline{At $t=t_k$ for $k\in\mathbb{N}$.} Let
			\begin{equation}
				\Delta(k)=\alpha \Delta(k-1),  \label{eq:delta_k_def}
			\end{equation}
		where the zooming factor  $\alpha \in (0,1)$ is a design parameter. We define the zoomed-in parameter set $\Theta(k)$ as
		 	\begin{equation}
				\Theta(k) := \mathcal{H}(\hat{p}(t_k^{-}),\Delta(k)) \cap \Theta(k-1) \cap \dots \cap \Theta(0). \label{eq:theta_k_def}
			\end{equation}
		   The sampled parameter set $\widehat{\Theta}(k):=\{p_{1}(k),\dots,p_{N}(k)\}$ consists of $N$ points which are selected such that (\ref{eq:dyn-sampling}) is verified.		
\end{itemize}


\noindent The dynamic sampling policy leads to the following changes in two components of the supervisory observer, namely the multi-observer \eqref{eq:observer_i} and the implementation of the monitoring signals \eqref{eq:monitoring_signal} described in Section \ref{sec:framework_static}. In the multi-observer, a state observer is designed for each sampled parameter $p_i(t)$, $i\in\{1,\dots,N\}$ for $t\in[t_k,t_{k+1})$,
\begin{equation}
	\begin{array}{rll}
		\dot{\hat{x}}_{i} & = & \hat{f}(\hat{x}_{i},p_i,u,y), \qquad \forall t \in [t_k,t_{k+1}), \, k\in\mathbb{N}, \\
		\hat{y}_{i} & = & h(\hat{x}_{i}, p_i), \\
		\hat{x}_{i}(t_k) & = &   \hat{x}_{i}(t_k^{-}),
	\end{array} \label{eq:observer_i_dynamic}
\end{equation}
where $\hat{x}_{i}\in \mathbb{R}^{n_x}$ is the state estimate provided by observer $i$. Note that the state $\hat{x}_{i}$ does not jump at update times $t_k$, $k\in\mathbb{N}$. Secondly, the definition of the monitoring signals \eqref{eq:monitoring_signal} remains unchanged, but the implementation as linear filters become, for $i\in\{1,\dots,N\}$,
\begin{equation}
	\begin{array}{rll}
	\dot{\mu}_{i}(t) & = & -\lambda \mu_{i}(t) + |\tilde{y}_{i}(t)|_{\infty}^{2}, \qquad \forall t \in [t_k,t_{k+1}), \, k\in\mathbb{N}, \\
	\mu_{i}(t_k) & = & 0.
	\end{array} \label{eq:cost_filter_dynamic}
\end{equation}

\subsection{Convergence guarantees}\label{sec:results-dynamic}
We are now ready to state the main result of this section.
\begin{thm} \label{thm:main_dynamic}
	Consider system \eqref{eq:general_plant}, the multi-observer \eqref{eq:observer_i_dynamic}, the monitoring signals \eqref{eq:cost_filter_dynamic}, the selection criterion \eqref{eq:selection_criterion}, the parameter estimate \eqref{eq:estimated_parameter}, the state estimate \eqref{eq:estimated_state} and the dynamic sampling policy \eqref{eq:update_time_interval}-\eqref{eq:theta_k_def}. Suppose Assumptions  \ref{ass:stable_sys}-\ref{ass:PE_y_tilde} are satisfied. For any $\Delta_{\tilde{x}}$, $\Delta_x$, $\Delta_u>0$, any margins $\nu_{\tilde{x}}$, $\nu_{\tilde{p}}>0$ and zooming factor $\alpha \in (0,1)$, there exist $\overline{T},\bar{K}_{\tilde{x}}>0$, sufficiently large $T^{\star}>0$ and $N^{\star}\in\mathbb{N}_{\geq 1}$ such that for any $N \geq N^{\star}$ and $T_d \geq T^{\star}$, for all $(x(0),\tilde{x}_{i}(0)) \in \mathcal{H}(0,{\Delta_x}) \times \mathcal{H}(0,{\Delta_{\tilde{x}}})$ for $i\in\{1,\dots,N\}$ and for any $u \in \mathcal{M}_{\Delta_u}$ that satisfies Assumption \ref{ass:PE_y_tilde}, the following holds
	\begin{equation}
		\begin{array}{c}
		\left|\tilde{p}_{\sigma(t)}(t) \right|_{\infty}\leq \nu_{\tilde{p}}\,\, \,\,\forall t\geq \overline{T},\, \qquad \underset{t\to \infty}{\limsup}\left|\tilde{x}_{\sigma(t)}(t) \right|_{\infty} \leq \nu_{\tilde{x}}, \qquad \left| \tilde{x}_{\sigma(t)}(t) \right|_{\infty} \leq \bar{K}_{\tilde{x}},\,\forall t \geq 0. \label{eq:main_dynamic_results}
		\end{array}
	\end{equation} \hfill $\Box$
\end{thm}

\noindent Theorem \ref{thm:main_dynamic} states that the estimated parameters $\hat{p}$ and states $\hat{x}$ in \eqref{eq:estimated_parameter} and \eqref{eq:estimated_state}, are respectively ensured to converge to their true values $p^{\star}$ and $x$ within some selected margins $\nu_{\tilde{p}}$ and $\nu_{\tilde{x}}$. These guarantees are the same as those in Theorem \ref{thm:main_static}. Nevertheless, for a given set of $N$ observers, Theorem \ref{thm:main_dynamic} may only ensure a better accuracy of the estimates compared to Theorem \ref{thm:main_static}. Indeed, consider a number of observers $N$ (sufficiently large). On the interval $[0,t_{1}]$, the static and the dynamic schemes provide the same estimates. Then, at time $t_{1}$, the dynamic scheme will start the zoom-in procedure which may only reduce the estimation error on the parameter and thus on the state. Hence, the dynamic scheme may be used to reduce the number of observers needed to ensure the convergence of the parameter and state estimation error up to given margins, as illustrated in Section \ref{sec:example}.

\section{Applications} \label{sec:special}
In this section, we apply the results of Theorem \ref{thm:main_dynamic} to two case studies: stable linear systems and a class of nonlinear systems.

\subsection{Linear systems} \label{sub:special_linear}
We consider the linear system
\begin{eqnarray}
	\dot{x} & = & A(p^{\star}) x + B(p^{\star})u \nonumber \\
	y & = & C(p^{\star}) x, \label{eq:lin_plant}
\end{eqnarray}
where $x\in\mathbb{R}^{n_x}$, $u\in\mathbb{R}^{n_u}$, $y\in\mathbb{R}^{n_y}$, $p^{\star}\in\Theta \subset \mathbb{R}^{n_p}$. The matrices $A(p)$, $B(p)$ and $C(p)$ are continuous in $p$ on $\Theta$, the pair $(A(p),C(p))$ is detectable for any $p\in\Theta$ and $A(p^{\star})$ is Hurwitz, which ensures the satisfaction of Assumption \ref{ass:stable_sys}.

Each observer in \eqref{eq:observer_i} is designed as follows for $p_i\in\Theta$, $i\in\{1,\dots,N\}$,
\begin{eqnarray}
	\dot{\hat{x}}_{i} & = & A(p_i) \hat{x}_{i} + B(p_i) u + L(p_i)(C(p_i)\hat{x}_{i} - y) \nonumber \\
	\hat{y}_{i} & = & C(p_i) \hat{x}_{i}, \label{eq:lin_obs}
\end{eqnarray}
where $L(p_i)$ is such that $A(p_i) + L(p_i) C(p_i)$ is Hurwitz (this is always possible since $(A(p_{i}),C(p_{i}))$ is detectable). The proposition below shows that Assumption \ref{ass:obs_error_i} is satisfied.

\begin{prop} \label{prop:lin_obs_error}
	Consider the linear system \eqref{eq:lin_plant} and the observer \eqref{eq:lin_obs}. Assumption \ref{ass:obs_error_i} holds. \hfill $\Box$
\end{prop}

If the classical PE condition \eqref{eq:tilde_y_PE_classic} is guaranteed, Assumption \ref{ass:PE_y_tilde} is satisfied according to Proposition \ref{prop:classic_PE} (note that item 2 of Proposition \ref{prop:classic_PE} holds for the considered system). There exist results in the literature (see Chapter 6 in \cite{narendra1989stable}) which provide sufficient conditions to verify \eqref{eq:tilde_y_PE_classic} for linear systems as mentioned earlier. These results can be used, for instance, to design an input $u$ to system \eqref{eq:lin_plant} such that the inequality \eqref{eq:tilde_y_PE_classic} is satisfied. The proposition below directly follows from Proposition \ref{prop:classic_PE} and Theorems \ref{thm:main_static}-\ref{thm:main_dynamic}.

\begin{prop} \label{prop:lin_res}
	Consider system \eqref{eq:lin_plant} and state-observer \eqref{eq:lin_obs} and suppose \eqref{eq:tilde_y_PE_classic} holds. When the static sampling policy described in Section \ref{sec:framework_static} (respectively the dynamic sampling policy of Section \ref{sec:framework_dynamic}) is used, then the conclusions of Theorem \ref{thm:main_static} (respectively Theorem \ref{thm:main_dynamic}) hold. \hfill $\Box$
\end{prop}


\subsection{A class of nonlinear systems} \label{sub:special_nonlinear}
We consider the following class of nonlinear systems studied in \cite{arcak2001observer}
\begin{eqnarray}
	\dot{x} & = & A(p^{\star}) x + G(p^{\star}) \gamma(Hx) + B(p^{\star}) \phi(u,y) \nonumber \\
	y & = & C(p^{\star})x, \label{eq:non_lin_plant}
\end{eqnarray}
where $x\in\mathbb{R}^{n_x}$, $u\in\mathbb{R}^{n_u}$, $y\in\mathbb{R}^{n_y}$, $p^{\star}\in\Theta \subset \mathbb{R}^{n_p}$, $\gamma:\mathbb{R}^{n_p}\rightarrow \mathbb{R}^{n_\gamma}$ and $\phi:\mathbb{R}^{n_u} \times \mathbb{R}^{n_y} \rightarrow \mathbb{R}^{n_\sigma}$.
The matrices $A(p)$, $G(p)$ and $B(p)$ are continuous in $p$ on $\Theta$.
We assume that Assumption \ref{ass:stable_sys} holds which can be verified with the aid of Lyapunov-based tools \cite{khalil1996nonlinear2}. We now explain how to design a state-observer which ensures the satisfaction of Assumption \ref{ass:obs_error_i}. For that purpose, we make the following assumption on the vector field $\gamma$, like in \cite{arcak2001observer,chong2012robust,zemouche2009unified}.
\begin{assum} \label{ass:nonlinear_functions}
	For any $k\in\{1,\dots,n_\gamma\}$, there exist constants $a_{\gamma_k}\in\mathbb{R}$, $b_{\gamma_k}\in\mathbb{R}\backslash\{0\}$ such that the following holds
	\begin{equation}
		-\infty < a_{\gamma_k} \leq \frac{\partial \gamma_{k}(v_{k})}{\partial v_{k}} \leq b_{\gamma_k} < \infty, \qquad \forall v_{k}\in \mathbb{R},
	\end{equation}
	where $\gamma=(\gamma_{1},\dots,\gamma_{n_\gamma})$. \hfill $\Box$
\end{assum}

\noindent The following state-observer \cite{arcak2001observer,chong2012robust,zemouche2009unified} is designed for any $p_i \in \Theta$, $i\in\{1,\dots,N\}$,
\begin{eqnarray}
	\dot{\hat{x}}_{i} & = & A(p_i) \hat{x}_{i} + G(p_i) \gamma(H\hat{x}_{i} + K(p_{i})(C\hat{x}_{i}-y)) + B(p_{i}) \phi(u,y) + L(p_{i})(C(p_i)\hat{x}_{i}-y) \nonumber \\
	\hat{y}_{i} & = & C(p_i) \hat{x}_{i}, \label{eq:non_lin_obs}
\end{eqnarray}
where $K(p_{i})$ and $L(p_{i})$ are the observer matrices. These matrices are selected such that the inequality (\ref{eq:lmi_cc}) below holds.

\begin{prop} \label{prop:cc_obs_error}
	Consider system \eqref{eq:non_lin_plant} and state-observer \eqref{eq:non_lin_obs} for $i\in\{1,\dots,N\}$. Suppose the following holds.
\begin{itemize}
\item[1)] Assumption \ref{ass:nonlinear_functions} holds.
\item[2)] There exist real matrices $P_{i}=P_{i}^{T}>0$, $M_{i}=\textrm{diag}(m_{i1},\dots,m_{in_{p}})>0$ and scalars $\nu_{i}$, $\mu_{i}>0$ such that the following holds
	\begin{equation}
\begin{array}{lllll}
		\left[\begin{array}{ccc}
			\mathcal{A}(P_{i},L(p_{i}),\nu_i) & \mathcal{B}(P_i,M_i,K(p_{i})) & P_i \\
			\star & \mathcal{E}(M_i) & 0  \\
			\star & \star & -\mu_{i}\mathbb{I}
		\end{array} \right] & \leq & 0, \label{eq:lmi_cc}
\end{array}	
\end{equation}
	where $\mathcal{A}(P_{i},L(p_{i}),\nu_i)=P_{i}\big(A(p_{i})+L(p_i)C(p_i)\big)+\big(A(p_{i})+L(p_i)C(p_i)\big)^{T}P_{i}+\nu_{i}\mathbb{I}$, \\ $\mathcal{B}(P_i,M_i,K(p_{i}))=P_{i}G(p_i)+\big(H+K(p_{i})C(p_i)\big)^{T} M_i$ and $\mathcal{E}(M_i)=-2M_i \textrm{diag}\left(\frac{1}{b_{\gamma_1}},\dots,\frac{1}{b_{\gamma_{n_p}}} \right)$.
\end{itemize}
Then Assumption \ref{ass:obs_error_i} is satisfied. \hfill $\Box$
\end{prop}

Note that inequality \eqref{eq:lmi_cc} is considered a linear matrix inequality (LMI) in $P_i$, $P_i L(p_i)$, $M_i K(p_i)$, $M_i$ and $\mu_{i}$. Therefore, \eqref{eq:lmi_cc} can be solved using efficient software such as the LMI solvers in MATLAB.

\begin{rem} 
	Other than the circle criterion based observer \eqref{eq:non_lin_obs}, any exponential observer\footnote{An exponential observer has a corresponding state estimation error system whose equilibrium is exponentially stable when there is no parameter mismatch, i.e. $\tilde{p}_{i}=0$. } for which its convergence is ensured using a quadratic Lyapunov function also satisfy Assumption 3 provided that the right hand sides of \eqref{eq:observer_i} are continuous in the parameter $p$. 
\end{rem}

We assume that the PE condition stated in Assumption \ref{ass:PE_y_tilde} is satisfied. PE properties for nonlinear systems are studied in \cite{panteley2001relaxed} and may be used in conjunction with Proposition \ref{prop:classic_PE} to verify Assumption \ref{ass:PE_y_tilde}. The following result is a direct application of Proposition \ref{prop:classic_PE} and Theorems \ref{thm:main_static}-\ref{thm:main_dynamic}.
\begin{prop} \label{prop:non_lin_res}
	Consider system \eqref{eq:non_lin_plant}, state-observer \eqref{eq:non_lin_obs} and suppose the following holds.
	\begin{enumerate}
		\item Assumptions \ref{ass:stable_sys} and \ref{ass:nonlinear_functions} and (\ref{eq:tilde_y_PE_classic}) hold.
		\item Condition \eqref{eq:lmi_cc} is feasible.
	\end{enumerate}
When the static sampling policy described in Section \ref{sec:framework_static} (respectively the dynamic sampling policy of Section \ref{sec:framework_dynamic}) is employed, then the conclusions of Theorem \ref{thm:main_static} (respectively of Theorem \ref{thm:main_dynamic}) hold. \hfill $\Box$
\end{prop}

\section{Illustrative example: A neural mass model} \label{sec:example}
In this section, we apply the results of Section \ref{sub:special_nonlinear} to estimate the synaptic gains (parameters) and the mean membrane potentials (states) of neuronal populations  of the neural mass model in \cite{jansen1995eeg}. The model in \cite{jansen1995eeg} describes the dynamics of a single cortical column by capturing the interactions between the pyramidal neurons, the excitatory and the inhibitory interneurons in a localised region of the cortex. It has been shown to realistically reproduce various patterns seen in the EEG recordings which is the measured output, such as alpha rhythms. Moreover, it may be used to generate more complex phenomena as shown in \cite{wendling2005interictal}.   

To write the model in the form of \eqref{eq:non_lin_plant}, we take the states\footnote[4]{According to the notation of \cite{jansen1995eeg}, $x=(y_0,y_3,y_1,y_4,y_2,y_5)$ and the parameter vector is taken to be $p^{\star}=(A,B)$.} to be $x=(x_{01},x_{02},x_{11},x_{12},x_{21},x_{22}) \in \mathbb{R}^{6}$, where $x_{01}$, $x_{11}$ and $x_{21}$ are the the membrane potential contributions of the pyramidal neurons, the excitatory and inhibitory interneurons respectively and $x_{02}$, $x_{12}$ and $x_{22}$ are their respective time-derivatives. The unknown vector of parameters$^{4}$ is  $p^{\star}=(p_1^{\star},p_2^{\star})$, where $p_1^{\star}$ and $p_2^{\star}$ represent the synaptic gains of the excitatory and inhibitory neuronal populations respectively. The vector of parameters $p^{\star}$ belongs to $\Theta := [4,8] \times [22,28]$ in agreement with \cite{jansen1995eeg}. The matrices in \eqref{eq:non_lin_plant} are defined as $C = \left(\begin{array}{cccccc} 0 & 0 & 1 & 0 & -1 & 0 \end{array}\right)$, $A = \textrm{diag}(A_a,A_a,A_b)$, where $A_a = \left(\begin{array}{cc} 0 & 1 \\ -a^{2} & -2a \end{array}\right)$, $A_b = \left(\begin{array}{cc}0 & 1 \\ -b^{2} & -2b  \end{array}\right)$, $G(p^{\star}) = \left(\begin{array}{cc} 0 & 0 \\ 0 & 0 \\ 0 & 0 \\ p_1 a c_2 & 0 \\ 0 & 0 \\ 0 & p_2 b c_4 \end{array}\right)$, $B(p^{\star}) = \left(\begin{array}{cc} 0 & 0 \\ p_1 a & 0 \\ 0 & 0 \\ 0 & p_1 a \\ 0 & 0 \\ 0 & 0 \end{array}\right)$\\ and $H = \left(\begin{array}{cccccc} c_1 & 0 & 0 & 0 & 0 & 0 \\ c_3 & 0 & 0 & 0 & 0 & 0 \end{array}\right)$.
The parameters $a$, $b$, $c_1$, $c_2$, $c_3$ and $c_4\in\mathbb{R}_{>0}$ are assumed to be known. The nonlinear terms in \eqref{eq:non_lin_plant} are $\gamma = (S,S)$ and $\phi(u,y)=(S(y),u)$. The function $S$ denotes the sigmoid function $S(v):=\frac{2e_0}{1+e^{r(v_0-v)}}$ for $v\in\mathbb{R}$, with known constants $e_0$, $v_0$, $r\in\mathbb{R}_{>0}$. For a detailed description of the model and its parameters, see \cite{jansen1995eeg}. The model has uniformly bounded solutions for all initial conditions and bounded input $u\in \mathcal{M}_{\Delta_u}$, because the matrix $A$ is Hurwitz and the nonlinearity $S$ in $\gamma$ and $\phi$ is bounded. Therefore, the model satisfies Assumption \ref{ass:stable_sys}. Furthermore, by the definition of the nonlinearity $S$ above, Assumption \ref{ass:nonlinear_functions} is satisfied.

We perform simulations with the model initialised at $x(0)=10 \times 1_{6\times 1}$, $p^{\star}=(p^{\star}_1,p^{\star}_2)=(6.5,25.5)$. Since the model can be written in the form of \eqref{eq:non_lin_plant}, we design the state-observers \eqref{eq:observer_i} of the form \eqref{eq:non_lin_obs}. The gains $K(p_i)$ and $L(p_i)$ are obtained by solving \eqref{eq:lmi_cc}. We implement the supervisory observer with both the static and dynamic sampling policies of Sections \ref{sec:framework_static} and \ref{sec:framework_dynamic}. The following design parameters are chosen: $\lambda = 0.005$ (for the monitoring signal in Section \ref{sec:monitoring_signal}), the zooming factor $\alpha = 0.8$ and the sampling interval $T_d=10$s (for the dynamic sampling policy in Section \ref{sec:framework_dynamic}). 

In Figures \ref{fig:parameter_estimates}-\ref{fig:res}, we compare the performance of the supervisory observer with and without dynamic sampling of parameters for $N=N_A \times N_B = 5\times 5$ observers, where $N_A$ and $N_B$ are the number of samples taken in the set of possible $p_1^{\star}$ and $p_2^{\star}$ values, respectively. Both policies allow us to estimate the states and the parameter up to some margin of the true values. The dynamic policy also gives better results in view of Figure \ref{fig:res}. To illustrate the workings of the dynamic sampling policy, we provide snapshots of the sampled parameters for $t\in[0,150)$s at each update time $t_k$ for $k\in\{0,1,5,9,10,14\}$ in Figure \ref{fig:55_10s_sample_1}. Note that for $t\geq 60$s, the plant parameter is no longer in the zoomed parameter set $\Theta(6)$, a phenomenon that is expected (see Section V). Nevertheless, the parameter estimate has converged to a desired neighbourhood of the plant parameter $p^{\star}$ for $t\geq 60$s.  

\begin{figure}[h!]
	\begin{center}
		\subfigure[Dynamic sampling with sampling interval $t_{k+1}-t_{k}=10$s.]{\includegraphics[width=8.5cm]{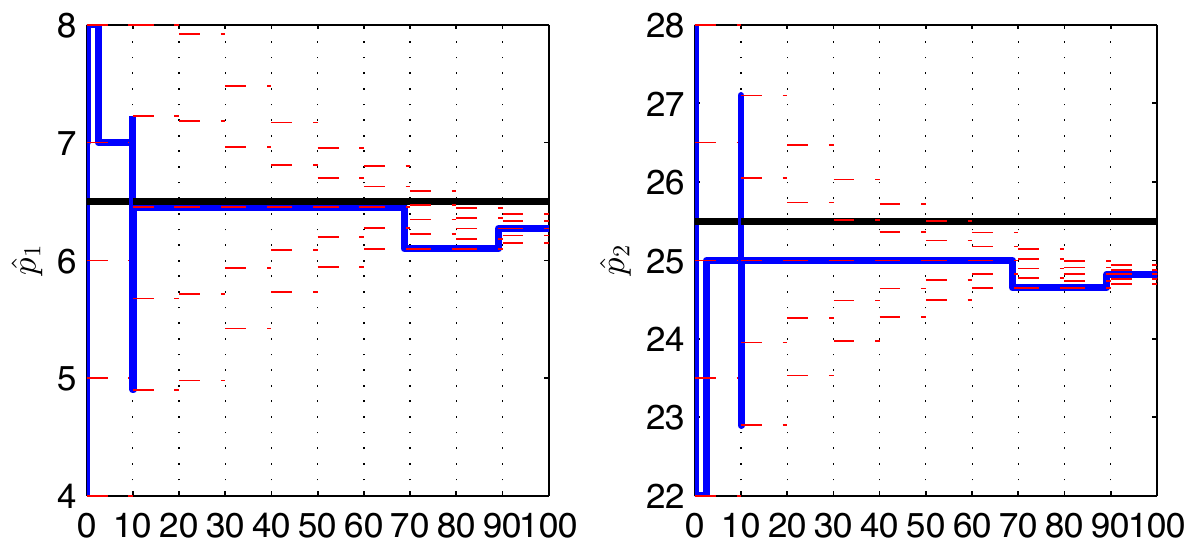}} 
		\subfigure[Static sampling]{\includegraphics[width=8.5cm]{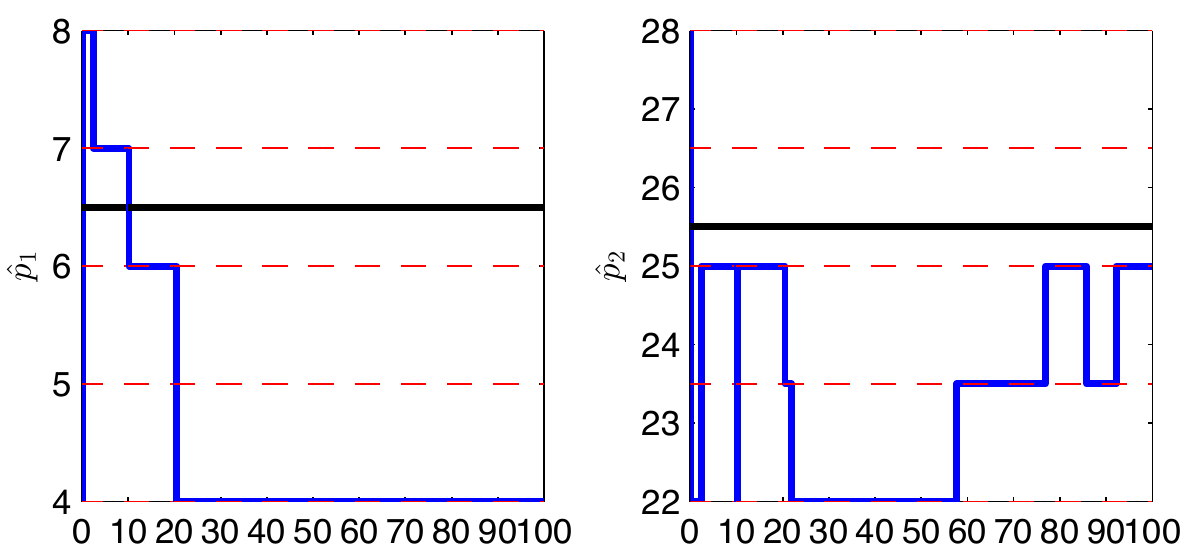}}
	\end{center} \vspace{-1em}
	\caption{ Estimated parameter $\hat{p}$ and the sampled parameter values. Legend: red dashed lines denote the chosen parameter values $p_i$, $i\in\{1,\dots,N\}$ and the vertical, dotted lines indicate the time of updating the parameter samples, $t_k$, $k\in\mathbb{N}$. The black solid line indicates the system parameter values $p^{\star}$. \label{fig:parameter_estimates} } 
\end{figure}

\begin{figure}[h!]
	\begin{center}
		\includegraphics[width=8.5cm]{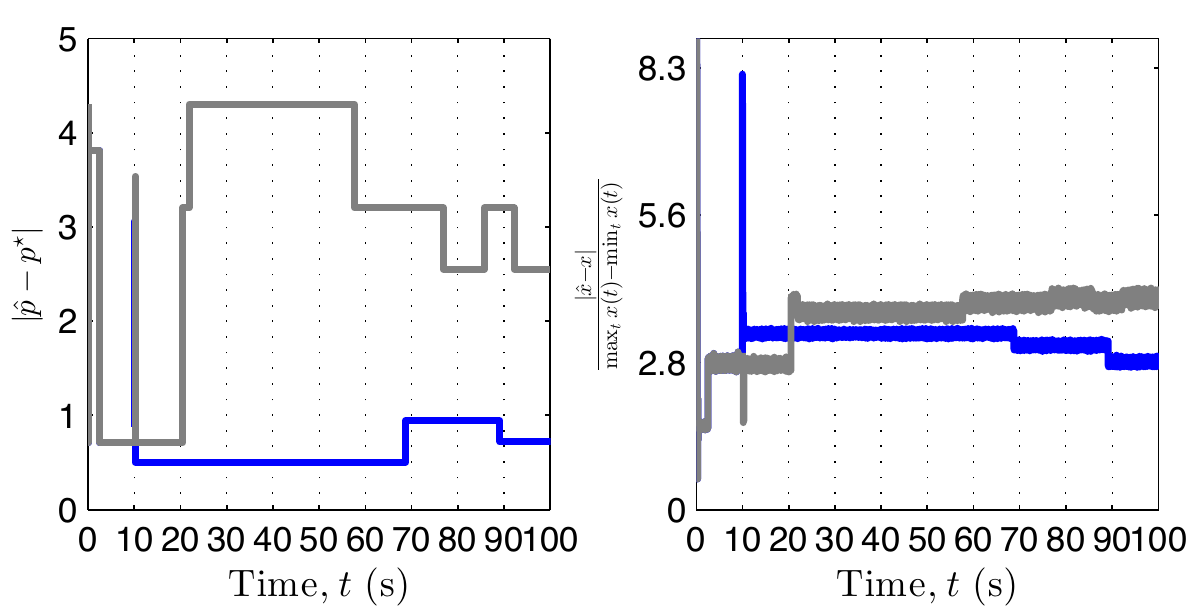}
	\end{center} \vspace{-1em}
	\caption{Norms of the parameter and state estimation errors obtained using the dynamic sampling policy with sampling interval $T_d = 10s$ (blue line) and the static sampling policy (grey line). \label{fig:res} } 
\end{figure}

\begin{figure}[h!]
	\begin{center}
		\includegraphics[width=4cm]{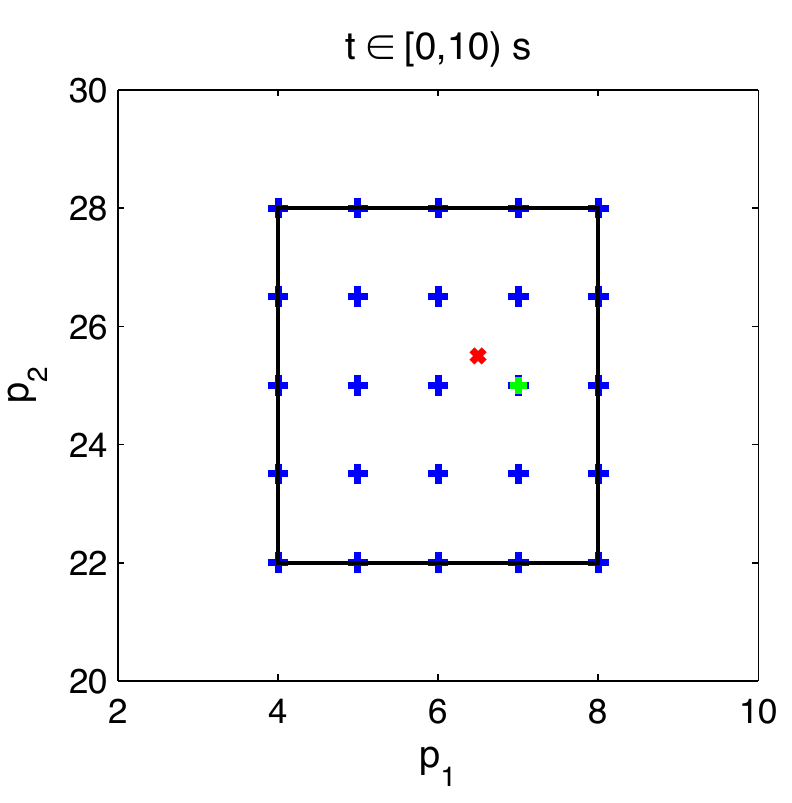}
		\includegraphics[width=4cm]{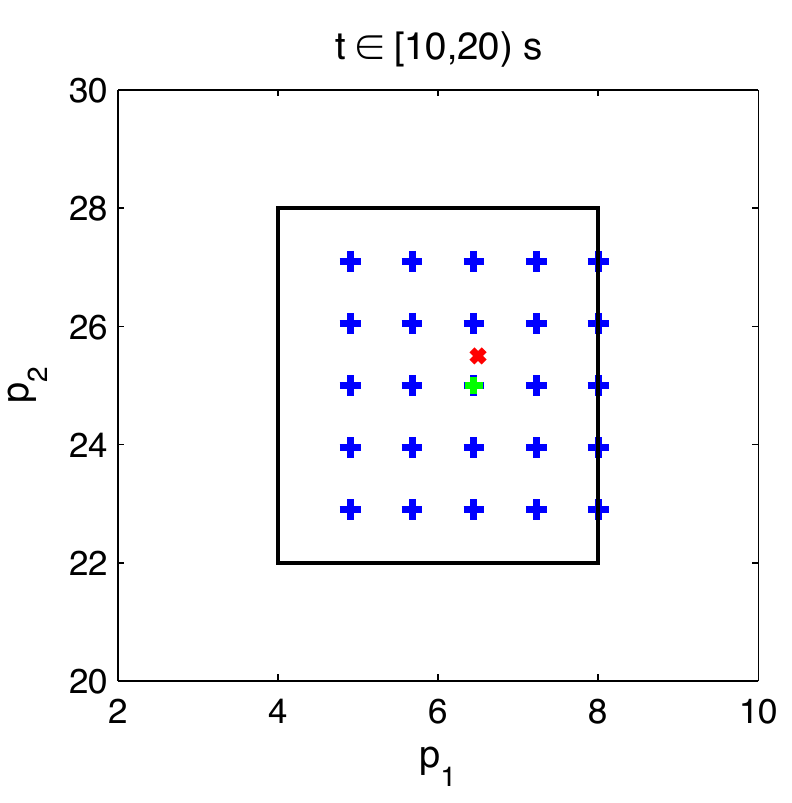}
		\includegraphics[width=4cm]{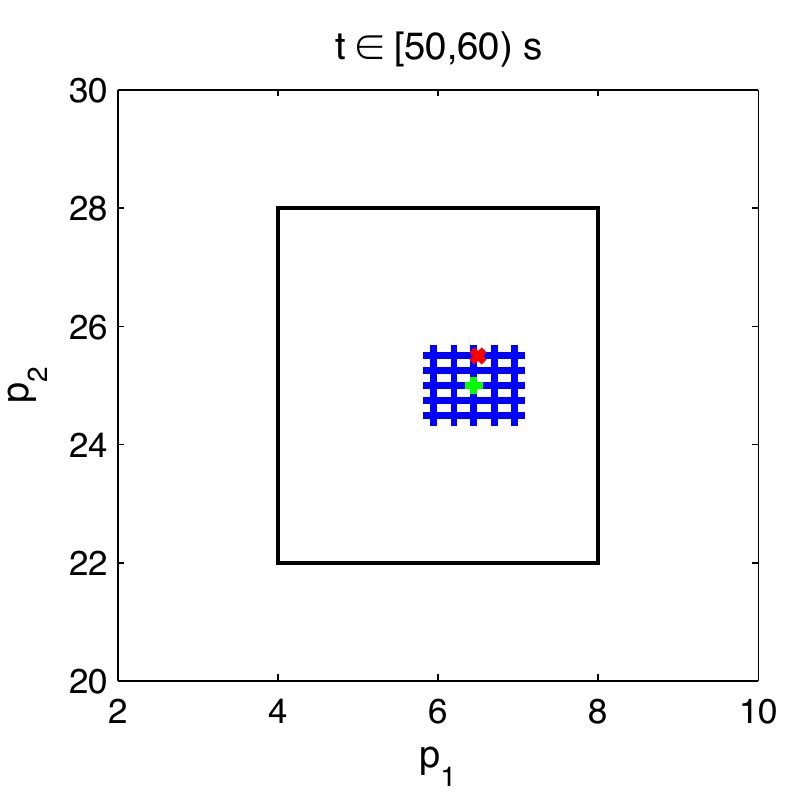}
		\includegraphics[width=4cm]{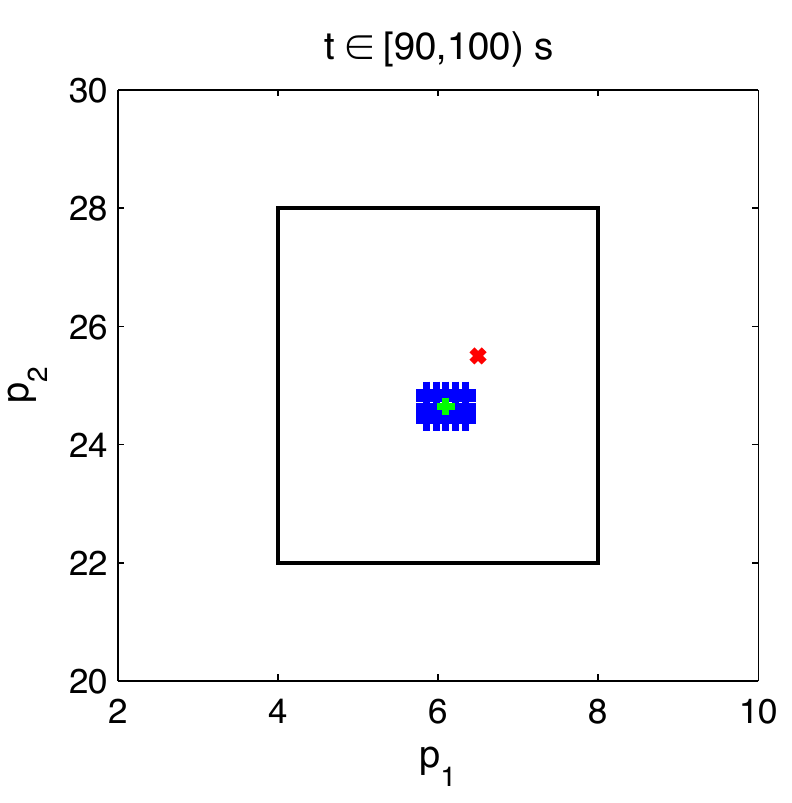}
		\includegraphics[width=4cm]{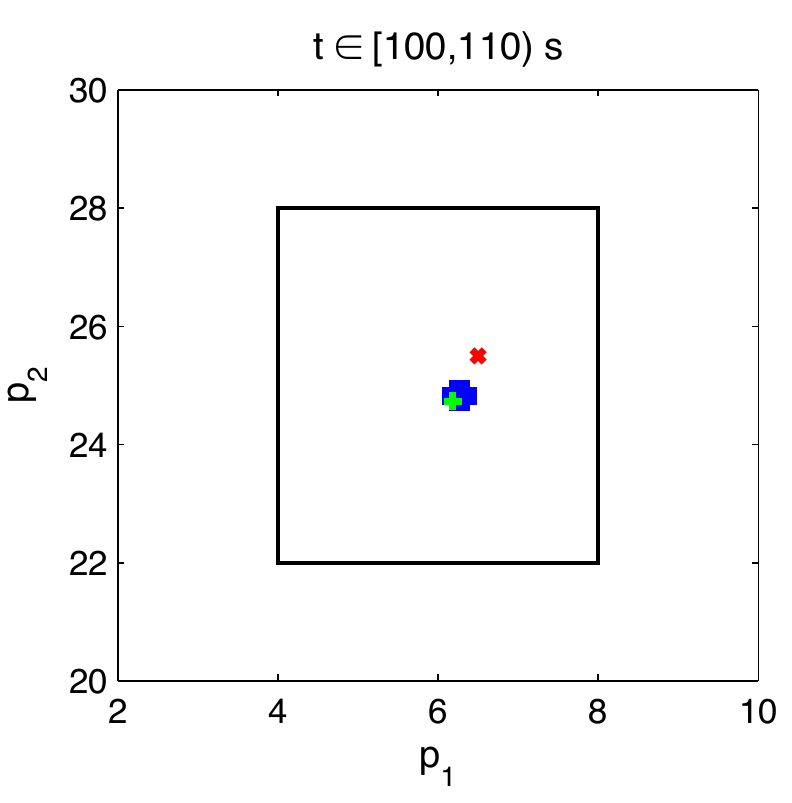}
		\includegraphics[width=4cm]{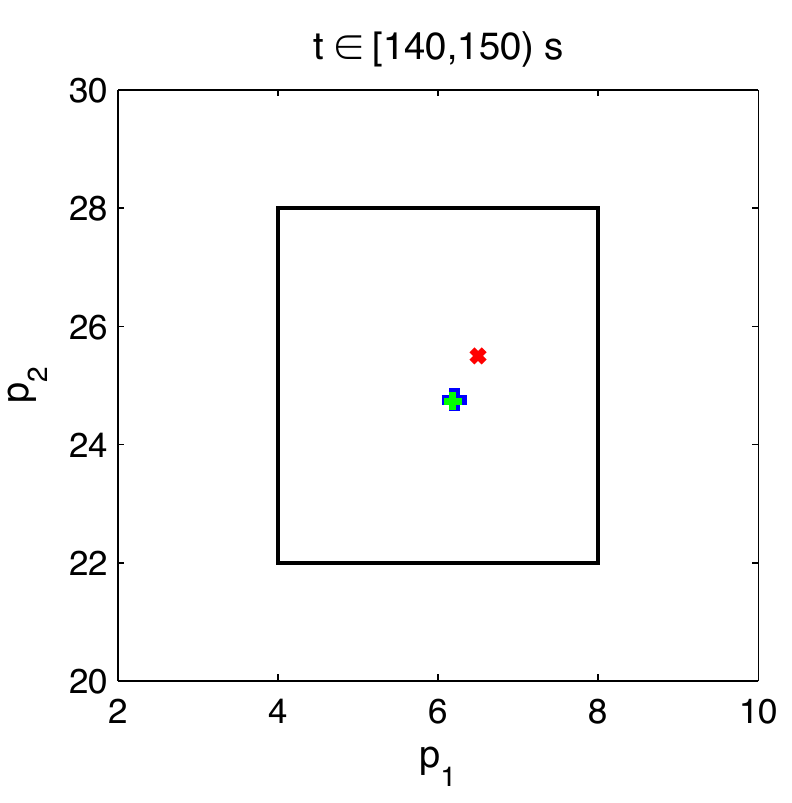}
	\end{center} \vspace{-2em}
	\caption{The dynamic sampling of the parameter set at some update instant $t_k$ for $(N_A,N_B)=(5,5)$ with sampling interval $t_{k+1}-t_{k}=10$s. Legend: {\color{red}$\times$} is the true parameter $p^{\star}$; {\color{blue}$+$} are the sampled parameters;  {\color{green} $+$} is the final selected parameter in the interval $t\in[t_k,t_{k+1})$; the known parameter set $\Theta$ is drawn with solid black lines.  \label{fig:55_10s_sample_1}  }
\end{figure}

To investigate the impact of the number of state-observers on the convergence results, we consider three values for $N=N_A \times N_B$: $2\times2$, $4\times 4$ and $5 \times 5$.  Table \ref{tab:static_dynamic_N} summarises the results obtained. We see that the supervisory observer with the dynamic policy outperforms the static scheme, in terms of the accuracy of the parameter estimate. This does not necessarily translate to a smaller ultimate bound on the state estimation error as seen in the results obtained for the static sampling policy. This can be explained as follows. Firstly, the results of Sections \ref{sec:framework_static} and \ref{sec:framework_dynamic} hold for a sufficiently large number of observers $N$ and it may be the case that $N=4$ and $N=16$ are not large enough to satisfy the conditions of Theorems \ref{thm:main_static} and \ref{thm:main_dynamic}. Secondly, although each observer satisfies Assumption \ref{ass:obs_error_i}, these state-observers do not have the same quantitative robustness properties with respect to the parameter estimation error $\tilde{p}_{i}$. To be precise, $\bar{\gamma}_{\tilde{x}}$ of the individual state estimation error systems \eqref{eq:error_sys_general} are in general different. Thus, the decrease of $|\tilde{p}_{i}|$ may be compensated by larger $\bar{\gamma}_{\tilde{x}}$. It will therefore be interesting to develop observers which minimise $\bar{\gamma}_{\tilde{x}}$ in future work.

\begin{table}[h!]
	\begin{center}
		\begin{tabular}{lcccc}
			\toprule
			$N=N_A \times N_B$ & $2\times2$ & $4\times4$ & $5\times5$  \\
			\midrule
			Static policy: $ \left| \tilde{p}_{\sigma(t_f)}(t_f) \right|$ & $4.30$ & $3.80$ & $2.55$  \\
			Dynamic policy: $ \left| \tilde{p}_{\sigma(t_f)}(t_f) \right|$ & $2.66$ & $1.04$ & $0.72$  \\
			Static policy: $ \frac{\left| \tilde{x}_{\sigma(t_f)}(t_f) \right|}{\max_{t} |x(t)| - \min_{t}|x(t)|}$ & $4.01$ & $1.73$ & $4.64$  \\
			Dynamic policy: $\frac{\left| \tilde{x}_{\sigma(t_f)}(t_f) \right|}{\max_{t\in[0,t_f]} |x(t)| - \min_{t\in[0,t_f]}|x(t)|}$ & $4.26$ & $3.47$ & $3.22$  \\
			\bottomrule
			The final time of simulation $t_f$ is $100$s.
		\end{tabular}
		\caption{Numerical results for increasing number of observers $N$ such that $d(p^{\star},\widehat{\Theta})$ decreases in the static case. \label{tab:static_dynamic_N}} 		
	\end{center}
	
\end{table}

\section{Conclusions} \label{sec:conclusion}
The main contribution of this paper is the design of schemes for the parameter and state estimation of nonlinear continuous-time systems with convergence guarantees. The proposed approach builds upon recent advances on supervisory control. It allows us to treat nonlinear systems provided robust state observers with respect to parameter errors can be designed and a given PE condition holds. We have shown that the parameter estimates converge in finite-time to the true parameter with any desired accuracy and that the norm of the state estimation error converges to the origin up to an adjustable margin, by taking a sufficiently large number of state-observers. This scheme may be computationally intense in some cases. This motivated the introduction of a dynamic sampling policy of the parameter set which may be used to ensure the same convergence guarantees with less number of observers. We have shown how these results can be applied to linear systems using Luenberger observers and to a class of nonlinear systems using circle criterion based observers. To illustrate the applicability of the supervisory observer, the schemes are applied to a neural mass model to estimate the synaptic gains and the mean membrane potentials of a cortical column. 

This work can be extended along two important directions. Firstly, the inclusion of measurement noises and model uncertainties. We expect the proposed supervisory observers to be well-equipped to handle this scenario provided the state-observers of the multi-observer unit are robust in an appropriate sense. Secondly, we mentioned in Section \ref{sec:framework_static} that prior information on the localization of the parameters within the parameter set in terms of probability distributions may be used to heuristically select the parameters value in the known parameter set. It would be interesting to assume the existence of such a distribution and to revisit the results of this paper to obtain (stochastic) convergence guarantees which would be potentially stronger than those currently ensured by the available nonlinear Kalman filtering techniques.


\appendix
\subsection{Proof of Proposition \ref{prop:classic_PE}}
For any $i\in\{1,\ldots,N\}$, $p_i\in\Theta$ and $\Theta$ is compact, hence $p_i - p^{\star}$ belongs to some compact set $\widetilde{\Theta}$. Let $p_{i}\in\Theta$ with $i\in\{1,\ldots,N\}$ such that $\tilde{p}_{i}\neq 0$, $\Delta_{\tilde{x}},\Delta_{x},\Delta_{u}>0$, $\tilde{x}_{i}(0)\in\mathcal{H}(0,\Delta_{\tilde{x}})$,
$x(0)\in\mathcal{H}(0,\Delta_{x})$ and $u\in\mathcal{M}_{\Delta_{u}}$. We denote the solution to the state error system \eqref{eq:error_sys_general} at time $t\geq 0$ as $\xi(t,\tilde{p}_i)$, where we have omitted its dependence on $x$, $u$ and $\tilde{x}_{i}(0)$. Similarly, we denote the output to \eqref{eq:error_sys_general} at time $t\geq 0$ as $\eta(t,\tilde{p}_{i})$.

We note that \eqref{eq:tilde_y_PE_classic} can be expressed in a scalar form as follows for all $t\geq T_f$,
\begin{equation}
\int_{t-T_f}^{t} (\nu^{T}\eta(s,\tilde{p}_{i}))^{2} ds \geq \bar{\alpha}_{i}, \label{eq:PE_classic_scalar}
\end{equation}
where $\nu \in \mathbb{R}^{n_y}$ is any constant vector with $|\nu|=1$. Since $|\nu^{T}\eta(s,\tilde{p}_{i})| \leq |\nu| |\eta(s,\tilde{p}_{i})|=|\eta(s,\tilde{p}_{i})|$ and $|\eta(s,\tilde{p}_{i})| \leq \sqrt{n_{y}}|\eta(s,\tilde{p}_{i})|_{\infty}$, inequality \eqref{eq:PE_classic_scalar} implies that
	\begin{equation}
		n_{y}\int_{t-T_f}^{t} |\eta(s,\tilde{p}_{i})|_{\infty}^{2} ds \geq \bar{\alpha}_{i}. \label{eq:PE_norm}
	\end{equation}
Let $W:(t,\tilde{p}_{i})\mapsto n_{y}\int_{t-T_f}^{t} |\eta(s,\tilde{p}_{i})|_{\infty}^{2} ds$ for $t\geq T_f$ and $\tilde{p}_{i} \in \widetilde{\Theta}$. We show that $W$ is continuous in $t$ and $\tilde{p}_{i}$. Since $f$ and $\hat{f}$ are continuously differentiable, $\tilde{p}_{i} \in \widetilde{\Theta}$, $p^{\star}\in\Theta$, $u \in \mathcal{M}_{\Delta_u}$ and $x \in \mathcal{M}_{K_x}$ (by Assumption \ref{ass:stable_sys}), $F_{i}$ is locally Lipschitz in $\xi$, uniformly in $\tilde{p}_{i}$, $p^{\star}$, $u$ and $x$ using similar arguments as in the proof of Lemma 2.3 of \cite{khalil1996nonlinear2}. Hence, $\xi(t,\tilde{p}_{i})$ is continuous in $t$ and $\tilde{p}_{i}$ by Theorem 2.6 in \cite{khalil1996nonlinear2}. As a consequence, we deduce that $\eta(t,\tilde{p}_{i})$ is also continuous in $t$ and $\tilde{p}_{i}$ by using the fact that $H$ is continuous in view of item 2) of Proposition \ref{prop:classic_PE}. Therefore, we have that $w:(s,\tilde{p}_{i})\mapsto n_{y}|\eta(s,\tilde{p}_{i})|^{2}_{\infty}$ is continuous in $s$ and $\tilde{p}_{i}$. As a consequence, $w$ is  uniformly continuous on $[t-T_f,t]\times\widetilde{\Theta}$ with $t\geq T_f$. Thus, given $\epsilon > 0$, there exists $\delta=\delta(\epsilon) >0$ such that for every pair of points $z=(s,\tilde{p}_{i})$ and $z'=(s',\tilde{p}_{i}')$ such that $|z-z'|<\delta$,  $|w(s,\tilde{p}_{i})-w(s',\tilde{p}_{i}')|<\epsilon$. Let $\epsilon>0$ and we fix the corresponding constant $\delta>0$. If $|\tilde{p}_{i}-\tilde{p}_{i}'| < \delta$, $\left| \int_{t-T_f}^t w(s,\tilde{p}_{i}) ds - \int_{t-T_f}^{t} w(s,\tilde{p}_{i}) ds \right| \leq \int_{t-T_f}^{t} \left| w(s,\tilde{p}_{i}) - w(s,\tilde{p}_{i}') \right| ds \leq T_f \epsilon$. Therefore, $W(t,\cdot)$ is continuous on $\widetilde{\Theta}$ for any $t\geq T_{f}$. On the other hand, the continuity of $W(\cdot,\tilde{p}_{i})$ for any $\tilde{p}_{i}\in\widetilde{\Theta}$ follows from the continuity of $w(\cdot,\tilde{p}_{i})$.

	For any $t\geq T_{f}$ and any $\tilde{p}_{i}\in\widetilde{\Theta}\backslash\{0\}$, $W(t,\tilde{p}_{i}) \geq \bar{\alpha}_{i}>0$ in view of (\ref{eq:PE_norm}). Using the fact that $W$ is continuous, we deduce that there exists a continuous and positive definite function $\widetilde{W}:\widetilde{\Theta}\to\mathbb{R}_{\geq 0}$ such that $\widetilde{W}(\tilde{p}_{i})\leq W(t,\tilde{p}_{i})$ for any $t\geq T_f$ and $\tilde{p}_{i}\in\widetilde{\Theta}$. By applying Lemma 3.5 in \cite{khalil1996nonlinear2} , we derive that there exists $\tilde{\alpha}_{\tilde{y}}\in\mathcal{K}_{\infty}$ such that $\widetilde{W}(\tilde{p}_{i})\geq \tilde{\alpha}_{\tilde{y}}(|\tilde{p}_{i}|)$ for any $\tilde{p}_{i}\in\widetilde{\Theta}$. Hence (\ref{eq:tilde_y_PE}) holds with $\alpha_{\tilde{y}}(s)=n_{y}^{-1}\tilde{\alpha}_{\tilde{y}}(s)$ (where we use the fact that $|\tilde{p}_i|\geq |\tilde{p}_i|_{\infty}$ for any $\tilde{p}_{i}$). \hfill $\Box$

	\subsection{Proof of Theorem \ref{thm:main_static}}
	We first prove that the state error systems \eqref{eq:error_sys_general} and the monitoring signals \eqref{eq:monitoring_signal} have desirable properties. Namely, the state error system \eqref{eq:error_sys_general} is shown to satisfy a local ISS property with respect to the parameter error $\tilde{p}_{i}$, $i\in\{1,\ldots,N\}$, in Lemma \ref{lem:bound_tilde_x}. This property, along with Assumption \ref{ass:PE_y_tilde}, are the key ingredients that allow us to show that the monitoring signals are lower and upper bounded by functions of the parameter error $\tilde{p}_{i}$, $i\in\{1,\ldots,N\}$, that are strictly increasing with $|\tilde{p}_{i}|$ in Lemma \ref{lem:bound_mu}. These lemmas are used to conclude the results of Theorem \ref{thm:main_static}.

	\begin{lem} \label{lem:bound_tilde_x}
		Consider  system \eqref{eq:general_plant} and the state error system \eqref{eq:observer_i} under Assumptions \ref{ass:stable_sys}-\ref{ass:obs_error_i}. There exist constants $\bar{k}$, $\bar{\lambda}>0$ such that for any $\Delta_{\tilde{x}},\Delta_{x},\Delta_{u}>0$ there exists a class $\mathcal{K}_{\infty}$ function $\bar{\gamma}_{\tilde{x}}$ such that for any $p_i\in\Theta$, $i\in\{1,\dots,N\}$, $\tilde{x}_{i}(0)\in  \mathcal{H}(0,{\Delta_{\tilde{x}}})$, $x(0) \in \mathcal{H}(0,\Delta_{x})$, $u\in \mathcal{M}_{\Delta_u}$, the corresponding solution satisfies
		\begin{equation}
			|\tilde{x}_{i}(t)|_{\infty} \leq \bar{k}\exp(-\bar{\lambda} t)|\tilde{x}_{i}(0)|_{\infty} + \bar{\gamma}_{\tilde{x}}(|\tilde{p}_{i}|_{\infty}), \qquad \forall t\geq 0. \label{eq:xi_iss}
		\end{equation} \hfill $\Box$
	\end{lem}
	\begin{proof}[Proof of Lemma \ref{lem:bound_tilde_x}]
		Let $p_i\in\Theta$, where $i\in\{1,\dots,N\}$, $V_i:\mathbb{R}^{n_x}\to \mathbb{R}_{\geq 0}$, $\tilde{\gamma}:\mathbb{R}^{n_p}\times \mathbb{R}^{n_x} \times \mathbb{R}^{n_u}\to \mathbb{R}_{\geq 0}$,  $a_1$, $a_2$, $\lambda_{0}>0$  be generated by Assumption \ref{ass:obs_error_i}. Given $\Delta_{\tilde{x}}$, $\Delta_{x}$, $\Delta_u>0$, let $u\in \mathcal{M}_{\Delta_u}$ and $K_{x}>0$ be generated by Assumption \ref{ass:stable_sys}. Since $\tilde{\gamma}$ is a continuous function and $\tilde{\gamma}(0,x,u)=0$, for all $x\in\mathbb{R}^{n_x}$, $u\in\mathbb{R}^{n_u}$, $\tilde{\gamma}$ can always be upper bounded by a positive definite function $\tilde{\gamma}(\tilde{p}_{i})=\max_{|x|_{\infty}\leq K_{x}, |u|_{\infty}\leq \Delta_u},\tilde{\gamma}(\tilde{p}_{i},x,u)$. By Lemma 3.5 in \cite{khalil1996nonlinear2} and since the Euclidean and infinity norms are equivalent in $\mathbb{R}^{n_p}$, there exists a class $\mathcal{K}_{\infty}$ function $\bar{\gamma}$ such that
		\begin{equation}
			\tilde{\gamma}_{1}(\tilde{p}_{i}) \leq \bar{\gamma}(|\tilde{p}_{i}|_{\infty}). \label{eq:bar_gamma}
		\end{equation}
		We use \eqref{eq:lyap_iss_1}, \eqref{eq:lyap_iss_2} and \eqref{eq:bar_gamma} to obtain the following for all $|x|_{\infty}\leq K_{x}$ and $|u|_{\infty} \leq \Delta_{u}$
		\begin{equation}
			\frac{\partial{V}_{i}}{\partial \tilde{x}_{i}}  F_{i}(\tilde{x}_{i},x,\tilde{p}_{i},p^{\star},u) \leq -\lambda_{0} V_{i}(\tilde{x}_{i}) + \bar{\gamma}(|\tilde{p}_{i}|_{\infty}). \label{eq:V_i}
		\end{equation}
		By the comparison principle (Lemma 2.5 in \cite{khalil1996nonlinear2}), for any $\tilde{x}_i(0)\in \mathcal{H}(0,\Delta_{\tilde{x}})$, $x(0) \in \mathcal{H}(0,\Delta_{x})$, $u \in \mathcal{M}_{\Delta_u}$, the corresponding solution to \eqref{eq:general_plant}, \eqref{eq:observer_i} verifies
		\begin{equation}
			V_{i}(\tilde{x}_{i}(t)) \leq \exp(-\lambda_{0}t) V_{i}(\tilde{x}_{i}(0)) + \frac{1}{\lambda_{0}}\bar{\gamma}(|\tilde{p}_{i}|_{\infty}), \label{eq:vv}
		\end{equation}
		 for $t \geq 0$. We use \eqref{eq:lyap_iss_1} and the fact that for any $a$, $b\geq 0$, $\sqrt{a+b} \leq \sqrt{a} + \sqrt{b}$ to obtain \eqref{eq:xi_iss} as desired with $\bar{k}:=\sqrt{\frac{a_{2}}{a_1}}$, $\bar{\lambda}:=\frac{\lambda_0}{2}$ and $\bar{\gamma}_{\tilde{x}}(r):=\sqrt{\frac{1}{a_{1}\lambda_{0}}}\sqrt{\bar{\gamma}(r)}$ for $r\geq 0$. 
	\end{proof}

	\begin{lem} \label{lem:bound_mu}	
		Consider system \eqref{eq:general_plant}, the state error system \eqref{eq:error_sys_general} and the monitoring signal \eqref{eq:monitoring_signal} under Assumptions \ref{ass:stable_sys}-\ref{ass:PE_y_tilde}. For any $\Delta_{\tilde{x}}$, $\Delta_{x}$, $\Delta_{u},\epsilon>0$, there exist class $\mathcal{K}_{\infty}$ functions $\underline{\chi}$ and $\bar{\chi}$ independent of $\epsilon$, a constant $T=T(\Delta_{\tilde{x}}, \Delta_{x}, \Delta_{u},\epsilon)>0$ such that for all $p_i\in\Theta$, $i\in\{1,\dots,N\}$,  $(x(0),\tilde{x}_{i}(0))\in \mathcal{H}(0,{\Delta_{x}}) \times \mathcal{H}(0,{\Delta_{\tilde{x}}})$, for some $u \in \mathcal{M}_{\Delta_u}$ that satisfies Assumption \ref{ass:PE_y_tilde}, the monitoring signal $\mu_{i}$ in \eqref{eq:monitoring_signal} satisfies
		\begin{equation}
			\underline{\chi}(|\tilde{p}_{i}|_{\infty}) \leq \mu_{i}(t) \leq \bar{\chi}(|\tilde{p}_{i}|_{\infty}) + \epsilon , \qquad \forall t  \geq T. \label{eq:mu_bound_x0}
		\end{equation}  \hfill $\Box$
	\end{lem}
	\begin{proof}[Proof of Lemma \ref{lem:bound_mu}]
		Let $\Delta_{\tilde{x}}$, $\Delta_{x}$, $\Delta_{u}$, $\epsilon>0$ and $p_i\in\Theta$, $i\in\{1,\dots,N\}$.
		\begin{itemize}
			\item Since $\Theta$ is compact, there exists $K_{\tilde{p}}>0$ such that for all $p\in\Theta$
			\begin{equation}
				|p-p^{\star}|_{\infty}\leq K_{\tilde{p}}, \label{eq:par_bound}
			\end{equation}
which implies that $|\tilde{p_{i}}|\leq K_{\tilde{p}}$.
		\item Let $\bar{k}$, $\bar{\lambda}>0$ and the class $\mathcal{K}_{\infty}$ function $\bar{\gamma}_{\tilde{x}}$ be generated from Lemma \ref{lem:bound_tilde_x}. By Lemma \ref{lem:bound_tilde_x}  and \eqref{eq:par_bound}, the solutions to \eqref{eq:general_plant}, \eqref{eq:error_sys_general} satisfy the following for all $(\tilde{x}_{i}(0),x(0))\in \mathcal{H}(0,\Delta_{\tilde{x}}) \times \mathcal{H}(0,{\Delta_{x}})$, $u\in \mathcal{M}_{\Delta_u}$ and $t\geq 0$,
		\begin{eqnarray}
			|\tilde{x}_{i}(t)|_{\infty} & \leq & \bar{k}\exp(-\bar{\lambda}t) |\tilde{x}_{i}(0)|_{\infty} + \bar{\gamma}_{\tilde{x}}(|\tilde{p}_{i}|_{\infty})
			 		 \leq  \bar{k} \Delta_{\tilde{x}} + \bar{\gamma}_{\tilde{x}}(K_{\tilde{p}})=:K_{\tilde{x}}. \label{eq:tilde_x_bound}
		\end{eqnarray}

		\item Since $h$ is continuously differentiable, $h$ is locally Lipschitz. Therefore, there exist constants $l_{\tilde{x}}$, $l_{\tilde{p}}>0$ such that for $x\in \mathcal{M}_{K_{x}}$ and $\tilde{x}_{i} \in \mathcal{M}_{K_{\tilde{x}}}$ with $K_{x},K_{\tilde{x}}>0$, the following holds for all $t\geq 0$
			{
			\begin{equation}
				\begin{array}{lll}
				|H(\tilde{x}_{i}(t), x(t), \tilde{p}_{i}, p^{\star})|_{\infty} & = &|H(\tilde{x}_{i}(t), x(t), \tilde{p}_{i}, p^{\star})-H(0,x(t),0,p^{\star})|_{\infty} \\
				& = &|h(\tilde{x}_{i}(t) + x(t),\tilde{p}_{i}+p^{\star}) - h(x(t),p^{\star})|_{\infty} \\
                & \leq & |h(\tilde{x}_{i}(t) + x(t),\tilde{p}_{i}+p^{\star}) - h(x(t),p^{\star})| \\
				& \leq & l_{\tilde{x}}|\tilde{x}_{i}(t)| + l_{\tilde{p}}|\tilde{p}_{i}|\\
				& \leq & \sqrt{n_{x}}l_{\tilde{x}}|\tilde{x}_{i}(t)|_{\infty} + \sqrt{n_{p}}l_{\tilde{p}}|\tilde{p}_{i}|_{\infty}, \label{eq:H_bound}
				\end{array}
			\end{equation}}				
			where we have used that $H(0,x(t),0,p^{\star})=0$ in view of the definition of $H$ in \eqref{eq:error_sys_general}.

		\item Let $\lambda > 0$ come from \eqref{eq:monitoring_signal}, we choose $\epsilon_{\tilde{x}}$, $\epsilon_{\mu}>0$ sufficiently small such that
				\begin{equation}
					\epsilon_{\mu} + \frac{4 n_{x} l_{\tilde{x}}^{2} \epsilon_{\tilde{x}}^{2}}{\lambda} = \epsilon.  \label{eq:def_eps}
				\end{equation}

			 Let $T_{\tilde{x}} > 0$ be sufficiently large such that
		\begin{equation}
			 \bar{k} \exp(-\bar{\lambda} t) \Delta_{\tilde{x}} \leq \epsilon_{\tilde{x}}, \qquad \forall t\geq T_{\tilde{x}}. \label{eq:beta_eps}
		\end{equation}
		Let $T_{\mu} \geq T_{\tilde{x}}$ be sufficiently large such that
				\begin{equation}
					\frac{1}{\lambda}\exp(-\lambda (T_{\mu}-T_{\tilde{x}})) (\sqrt{n_{x}}l_{\tilde{x}}K_{\tilde{x}} + \sqrt{n_{p}}l_{\tilde{p}} K_{\tilde{p}})^{2} \leq \epsilon_{\mu}. \label{eq:mu_eps}
				\end{equation}
		\item Let $T_f>0$ and the class $\mathcal{K}_{\infty}$ function $\alpha_{\tilde{y}}$ be generated by Assumption \ref{ass:PE_y_tilde}  such that for all $\tilde{x}_{i}(0)\in \mathcal{H}(0,\Delta_{\tilde{x}})$, for some $u\in \mathcal{M}_{\Delta_u}$, for all  $x\in\mathcal{M}_{K_{x}}$, the solution to \eqref{eq:error_sys_general} satisfies
				\begin{equation}
					\int_{t-T_f}^{t} |H(\tilde{x}_{i}(s), x(s), \tilde{p}_{i}, p^{\star})|_{\infty}^{2} ds \geq \alpha_{\tilde{y}}(|\tilde{p}_{i}|_{\infty}), \qquad \forall t\geq T_f. \label{eq:PE_H}
				\end{equation}
		\item We define the following
			\begin{equation}
				T := \max\{T_{\mu}, T_{\tilde{x}}, T_f\}, \label{eq:def_T}
			\end{equation}
			\begin{equation}
				\underline{\chi}(r):=  \exp\left(-\lambda T_{f} \right) \alpha_{\tilde{y}}(r) \qquad \forall r\geq 0 \label{eq:under_chi}
			\end{equation}
			\begin{equation}
				\bar{\chi}(r):=  \frac{4 n_{x} l_{\tilde{x}}^{2}}{\lambda} \bar{\gamma}_{\tilde{x}}^{2}(r) +  \frac{2 n_{p} l_{\tilde{p}}^{2}}{\lambda} r^{2} \qquad \forall r \geq 0. \label{eq:bar_chi}
			\end{equation}
		\end{itemize}
		Note that the class $\mathcal{K}_{\infty}$ functions $\underline{\chi}$ and $\bar{\chi}$ depend only on $\Delta_{\tilde{x}}$, $\Delta_{x}$, $\Delta_{u}$ and not on $\epsilon$.
		Let $t\geq T$, where $T$ is defined in \eqref{eq:def_T}. By definition of the monitoring signals in \eqref{eq:monitoring_signal},
		\begin{equation}
			\mu_{i}(t) =  \int_{0}^{t} \exp(-\lambda (t-s)) |H(\tilde{x}_{i}(s), x(s), \tilde{p}_{i}, p^{\star})|_{\infty}^{2} ds.
		\end{equation}
		We first establish the desired lower bound on $\mu_{i}$
		\begin{eqnarray}
			\mu_{i}(t) & = &  \int_{0}^{t-T_{f}} \exp(-\lambda (t-s)) |H(\tilde{x}_{i}(s), x(s), \tilde{p}_{i}, p^{\star})|_{\infty}^{2} ds \nonumber \\
									&& +  \int_{t-T_{f}}^{t} \exp(-\lambda (t-s)) |H(\tilde{x}_{i}(s), x(s), \tilde{p}_{i}, p^{\star})|_{\infty}^{2} ds  \\
									& \geq &  \int_{t-T_{f}}^{t} \exp(-\lambda (t-s)) |H(\tilde{x}_{i}(s), x(s), \tilde{p}_{i}, p^{\star})|_{\infty}^{2} ds.
		\end{eqnarray}
		As $s \mapsto \exp(\lambda s)$ is strictly increasing,
		\begin{equation}
			\mu_{i}(t) \geq  \exp(-\lambda T_{f}) \int_{t-T_{f}}^{t} |H(\tilde{x}_{i}(s), x(s), \tilde{p}_{i}, p^{\star})|_{\infty}^{2} ds.
		\end{equation}
		From \eqref{eq:PE_H}, \eqref{eq:under_chi}, since $t\geq T_{f}$,
		\begin{equation}
			\mu_{i}(t) \geq  \exp(-\lambda T_{f}) \alpha_{\tilde{y}}(|\tilde{p}_{i}|_{\infty}) = \underline{\chi}(|\tilde{p}_{i}|_{\infty}). \label{eq:lower_mu_claim}
		\end{equation}	
		We now obtain the desired upper bound of $\mu_{i}$
		\begin{eqnarray}
			\mu_{i}(t) & = & \int_{0}^{T_{\tilde{x}}} \exp(-\lambda (t-s)) |H(\tilde{x}_{i}(s), x(s), \tilde{p}_{i}, p^{\star})|_{\infty}^{2} ds \nonumber \\
									&& +  \int_{T_{\tilde{x}}}^{t} \exp(-\lambda (t-s)) |H(\tilde{x}_{i}(s), x(s), \tilde{p}_{i}, p^{\star})|_{\infty}^{2} ds.
		\end{eqnarray}
		Using \eqref{eq:H_bound} and the fact that for any $a$, $b\geq 0$, $(a+b)^{2} \leq 2a^{2} + 2b^{2}$, we obtain
		\begin{equation}
        \begin{array}{lllllll}
				\mu_{i}(t) & \leq & \frac{1}{\lambda}\exp(-\lambda(t-T_{\tilde{x}})) \left( \sup_{s\in [0,T_{\tilde{x}}]} |H(\tilde{x}_{i}(s), x(s), \tilde{p}_{i}, p^{\star})|_{\infty}^{2} \right) \\
						&& + \int_{T_{\tilde{x}}}^{t} \exp(-\lambda (t-s)) \Big( \sqrt{n_{x}}l_{\tilde{x}}|\tilde{x}_{i}(s)|_{\infty} + \sqrt{n_{p}}l_{\tilde{p}}|\tilde{p}_{i}|_{\infty} \Big)^{2} ds\\
						& \leq & \frac{1}{\lambda} \exp(-\lambda(t-T_{\tilde{x}})) \left( \sup_{s\in [0,T_{\tilde{x}}]} |H(\tilde{x}_{i}(s), x(s), \tilde{p}_{i}, p^{\star})|_{\infty}^{2} \right) \\
						&& + \int_{T_{\tilde{x}}}^{t} \exp(-\lambda (t-s)) \Big( 2n_{x}l_{\tilde{x}}^{2}|\tilde{x}_{i}(s)|_{\infty}^{2} + 2n_{p}l_{\tilde{p}}^{2}|\tilde{p}_{i}|_{\infty}^{2} \Big)  ds. \label{eq:mu_upper_1}
		\end{array}
\end{equation}
		We also have from Lemma \ref{lem:bound_tilde_x} and \eqref{eq:beta_eps} that, as $t\geq T_{\tilde{x}}$,
		\begin{equation}
			|\tilde{x}_{i}(t)|_{\infty} \leq \epsilon_{\tilde{x}} + \gamma_{\tilde{x}}(|\tilde{p}_{i}|_{\infty}),
		\end{equation}
		which implies that, using the fact that for any $a$, $b>0$, $(a+b)^{2} \leq 2a^{2} + 2b^{2}$
		\begin{equation}
			|\tilde{x}_{i}(t)|_{\infty}^{2} \leq 2\epsilon_{\tilde{x}}^{2} + 2\gamma_{\tilde{x}}^{2}(|\tilde{p}_{i}|_{\infty}). \label{eq:bound_tilde_2}
		\end{equation}
		Also, by \eqref{eq:tilde_x_bound} and \eqref{eq:H_bound}, we obtain
		\begin{equation}
        \begin{array}{llll}
			\sup_{s\in [0,T_{\tilde{x}}]} |H(\tilde{x}_{i}(s), x(s), \tilde{p}_{i}, p^{\star})|_{\infty}^{2}  & \leq &  \sup_{s\in [0,T_{\tilde{x}}]} (\sqrt{n_{x}}l_{\tilde{x}} |\tilde{x}_{i}(s)|_{\infty} + \sqrt{n_{p}}l_{\tilde{p}}|\tilde{p}_{i}|_{\infty}) ^{2}\\
       & = & (\sqrt{n_{x}}l_{\tilde{x}} K_{\tilde{x}} + \sqrt{n}_{p}l_{\tilde{p}}K_{\tilde{p}})^{2}. \label{eq:sup_1}
\end{array}		
\end{equation}
		Hence, from \eqref{eq:mu_upper_1} and in view of \eqref{eq:mu_eps}, \eqref{eq:bound_tilde_2} and \eqref{eq:sup_1},
		\begin{eqnarray}
			\mu_{i}(t) & \leq & \epsilon_{\mu} + \int_{T_{\tilde{x}}}^{t} \exp(-\lambda (t-s)) \Big( 4n_{x}l_{\tilde{x}}^{2}\epsilon_{\tilde{x}}^{2} + 4 n_{x}l_{\tilde{x}}^{2} \gamma_{\tilde{x}}^{2}(|\tilde{p}_{i}|_{\infty}) + 2n_{p}l_{\tilde{p}}^{2}|\tilde{p}_{i}|_{\infty}^{2} \Big)  ds  \nonumber\\
					& = &  \epsilon_{\mu} + \left( \int_{T_{\tilde{x}}}^{t} \exp(-\lambda (t-s)) ds \right) \Big( 4n_{x}l_{\tilde{x}}^{2} \epsilon_{\tilde{x}}^{2} + 4 n_{x}l_{\tilde{x}}^{2} \gamma_{\tilde{x}}^{2}(|\tilde{p}_{i}|_{\infty}) + 2n_{p}l_{\tilde{p}}^{2}|\tilde{p}_{i}|_{\infty}^{2} \Big). \label{eq:mu_int_1}
		\end{eqnarray}
		As $\int_{T_{\tilde{x}}}^{t} \exp(-\lambda (t-s)) ds = \frac{1}{\lambda} \big( 1- \exp(-\lambda (t-T_{\tilde{x}})) \big) \leq \frac{1}{\lambda}$. Therefore, from \eqref{eq:mu_int_1},
		\begin{eqnarray}
			\mu_{i}(t) & \leq & \epsilon_{\mu} + \frac{4 n_{x}l_{\tilde{x}}^{2} \epsilon_{\tilde{x}}^{2}}{\lambda} +  \frac{4 n_{x}l_{\tilde{x}}^{2}}{\lambda} \gamma_{\tilde{x}}^{2}(|\tilde{p}_{i}|_{\infty}) + \frac{2 n_{p}l_{\tilde{p}}^{2}}{\lambda}|\tilde{p}_{i}|_{\infty}^{2}.
		\end{eqnarray}
		By the definition of $\epsilon$ and $\bar{\chi}$ in \eqref{eq:def_eps} and \eqref{eq:bar_chi} respectively
		\begin{equation}
			\mu_{i}(t) \leq \epsilon + \bar{\chi}(|\tilde{p}_{i}|_{\infty}). \label{eq:upper_mu_claim}
		\end{equation}
		Therefore, \eqref{eq:mu_bound_x0} holds in view of \eqref{eq:lower_mu_claim} and \eqref{eq:upper_mu_claim}.
	\end{proof}
	
$\,$\\
\begin{proof}[Proof of Theorem \ref{thm:main_static}]
Given $\Delta_{\tilde{x}}$, $\Delta_{x}$, $\Delta_u > 0$ and $\nu_{\tilde{p}}$, $\nu_{\tilde{x}}>0$, we construct the ingredients needed.
\begin{itemize}
	\item Let $\nu_{1}>0$ be sufficiently small such that
			\begin{equation}
				\nu_1 \leq \bar{\gamma}_{\tilde{x}}^{-1}(\nu_{\tilde{x}}), \label{eq:nu1_bound}
			\end{equation}
			where $\bar{\gamma}_{\tilde{x}}$ is the class $\mathcal{K}_{\infty}$ function generated by Lemma \ref{lem:bound_tilde_x}.
	\item Let the class $\mathcal{K}_{\infty}$ functions $\bar{\chi}$ and $\underline{\chi}$ be generated by Lemma \ref{lem:bound_mu}. Choose $\epsilon \in \left(0,\underline{\chi}(\min\{\nu_1,\nu_{\tilde{p}}\}) \right]$ and
			\begin{equation}
				d^{\star} := \bar{\chi}^{-1}\left(\underline{\chi}(\min\{\nu_1,\nu_{\tilde{p}}\})-\epsilon \right). \label{eq:dstar_def}
			\end{equation}
\end{itemize}
Recall that the chosen monitoring signal at any time $t\geq 0$ is denoted as $\mu_{\sigma(t)}(t)$. By the definition of the selection criterion \eqref{eq:selection_criterion},
\begin{equation}
	\mu_{\sigma(t)}(t) \leq \mu_{j}(t), \qquad \forall t\geq 0,\,\,\forall j\in\{1,\dots,N\}. \label{eq:mu_ineq}
\end{equation}
Since \eqref{eq:mu_ineq} holds, we consider the monitoring signal with the smallest parameter estimation error, i.e. ${j^{\star}}=\arg\underset{j\in\{1,\dots,N\}}{\min}\left|\tilde{p}_{j}\right|_{\infty}$ and $|\tilde{p}_{j^{\star}}|=d(p^{\star},\widehat{\Theta})$ where $d(p^{\star},\widehat{\Theta})$ is defined in (\ref{eq:pt_set_dist}). Using Lemma \ref{lem:bound_mu}, there exists $T > 0$ such that
\begin{equation}
		\underline{\chi}(|\tilde{p}_{\sigma(t)}(t)|_{\infty}) \leq \mu_{\sigma(t)}(t) \leq \mu_{j^{\star}}(t) \leq \bar{\chi}(|\tilde{p}_{j^{\star}}|_{\infty}) + \epsilon, \qquad \forall t \geq T.
\end{equation}
Therefore
\begin{equation}\label{eq:tildep_f-prel}
	\left| \tilde{p}_{\sigma(t)}(t) \right|_{\infty} \leq \underline{\chi}^{-1}\left( \bar{\chi}\left( d(p^{\star},\widehat{\Theta}) \right) + \epsilon \right), \qquad \forall t \geq T.
\end{equation}

\noindent Recall from Section \ref{sec:static_sampling} that we sample the parameter set $\Theta$ in a manner such that $d(p^{\star},\widehat\Theta)\to 0$ as $N\to\infty$. Therefore, there is an $N^{\star}\in\mathbb{N}_{\geq 1}$ such that $d(p^{\star},\widehat{\Theta})\leq d^{\star}$ for all $N \geq N^{\star}$. We fix $N^{\star}$ and take $N\geq N^{\star}$. Hence, for all $t\geq T$
\begin{equation}
	\begin{array}{lll}
	\left| \tilde{p}_{\sigma(t)}(t) \right|_{\infty} & \leq & \underline{\chi}^{-1}\left( \bar{\chi}\left( d^{\star} \right) + \epsilon \right)  =  \min\{\nu_1,\nu_{\tilde{p}} \} \leq \nu_{\tilde{p}}.
\end{array} \label{eq:tildep_f}
\end{equation}

\noindent We now examine the state estimation error $\tilde{x}_{\sigma(t)}(t)$. For all $t\geq 0$,
\begin{equation}
	\left| \tilde{x}_{\sigma(t)}(t)\right|_{\infty} \leq \underset{i\in\{1,\dots,N\}}{\max}\left| \tilde{x}_{i}(t) \right|_{\infty}. \label{eq:bound_tilde_x_1}
\end{equation}
By Lemma \ref{lem:bound_tilde_x}, we have that for $i\in\{1,\dots,N\}$, for all $\tilde{x}_{i}(0) \in \mathcal{H}(0,{\Delta_{\tilde{x}}})$, $x(0)\in\mathcal{H}(0,{\Delta_{{x}}})$ and $u\in\mathcal{M}_{\Delta_u}$, the solution to (\ref{eq:general_plant}), \eqref{eq:error_sys_general} satisfies
\begin{equation}
	\left| \tilde{x}_{i}(t) \right|_{\infty} \leq \bar{k} \exp(-\bar{\lambda} t) |\tilde{x}_{i}(0)|_{\infty} + \bar{\gamma}_{\tilde{x}}(|\tilde{p}_{i}|_{\infty}), \label{eq:bound_tilde_x_2}
\end{equation}
where $\bar{k}$, $\bar{\lambda}>0$ and $\bar{\gamma}_{\tilde{x}}\in\mathcal{K}_{\infty}$. Like in the proof of Lemma \ref{lem:bound_mu}, since  $p^{\star}$, ${p}_{i} \in \Theta$, there exists $K_{\tilde{p}}>0$, which is independent of $i$, such that $|\tilde{p}_{i}|_{\infty}\leq K_{\tilde{p}}$. Hence, from \eqref{eq:bound_tilde_x_1} and \eqref{eq:bound_tilde_x_2}, we obtain for all $t\geq 0$
\begin{equation}
	\begin{array}{lll}
	\left| \tilde{x}_{i}(t) \right|_{\infty} & \leq & \bar{k} \Delta_{\tilde{x}} + \bar{\gamma}_{\tilde{x}}\left(\underset{i\in\{1,\dots,N\} }{\max}|\tilde{p}_{i}|_{\infty} \right) \\
	& = & \bar{k} \Delta_{\tilde{x}} + \bar{\gamma}_{\tilde{x}}\left(K_{\tilde{p}} \right) =:\bar{K}_{\tilde{x}}.
	\end{array} \label{eq:tilde_x_bound_final}
\end{equation}
Hence, we have from \eqref{eq:bound_tilde_x_1} and \eqref{eq:tilde_x_bound_final} that the solution to the chosen state estimation error system satisfies
\begin{equation}
	\left| \tilde{x}_{\sigma(t)}(t)\right|_{\infty} \leq \bar{K}_{\tilde{x}}, \qquad \forall t \geq 0. \label{eq:bound_tildex_done}
\end{equation}

\noindent Furthermore, we have from \eqref{eq:tildep_f} that the parameter error $\tilde{p}_{\sigma}$ converges to the hypercube centered at $0$ and of edge length $2\nu_{\tilde{p}}$ in finite-time. Let $\mathcal{S}:=\{ i\in\{1,\dots,N\}\, :\, |\tilde{p}_{i}| \leq \min\{ \nu_{\tilde{p}}, \nu_1 \} \}$. Therefore, as $t\geq T$
\begin{equation}
	|\tilde{x}_{\sigma(t)}(t)|_{\infty} \leq \max_{i\in \mathcal{S}} |\tilde{x}_{i}(t)|_{\infty}.
\end{equation}
Consequently, we have from \eqref{eq:bound_tilde_x_2} that
\begin{eqnarray}
	\limsup_{t\rightarrow \infty} |\tilde{x}_{\sigma(t)}(t)|_{\infty} & \leq & \limsup_{t\rightarrow \infty} \max_{i\in \mathcal{S}} |\tilde{x}_{i}(t)|_{\infty} \leq \max_{i\in \mathcal{S}} \bar{\gamma}_{\tilde{x}}(|\tilde{p}_{i}|_{\infty}) \leq \bar{\gamma}_{\tilde{x}}(\min\{ \nu_{\tilde{p}},\nu_1\}) \leq \bar{\gamma}_{\tilde{x}}(\nu_1). \nonumber 
\end{eqnarray}
By \eqref{eq:nu1_bound}, we obtain
\begin{equation}
	\limsup_{t\rightarrow \infty} |\tilde{x}_{\sigma(t)}(t)|_{\infty} \leq \nu_{\tilde{x}}. \label{eq:sup_tilde_x_done}
\end{equation}
Finally, we have shown \eqref{eq:main_static_results} in view of \eqref{eq:tildep_f}, \eqref{eq:bound_tildex_done}  and \eqref{eq:sup_tilde_x_done}.
\end{proof}


\subsection{Proof of Theorem \ref{thm:main_dynamic}}

Let $\Delta_{\tilde{x}},\Delta_x,\Delta_u,\nu_{\tilde{x}},\nu_{\tilde{p}}>0$ and $\alpha \in (0,1)$. The solutions to \eqref{eq:general_plant} and \eqref{eq:observer_i_dynamic} are unique and are defined for all positive time, for any initial condition and any input $u$ in view of Sections \ref{sec:prob} and \ref{sec:multiobs}. On the other hand, $p_{i}(t)\in\Theta$ for any $i\in\{1,\ldots,N\}$ and $t\geq 0$ in view of Section \ref{sec:description-dynamic}. Hence, we will be able to invoke the results of Lemmas \ref{lem:bound_tilde_x}-\ref{lem:bound_mu} to derive the desired result.

Let $\nu_{1}>0$ be sufficiently small such that
\begin{equation}
\nu_1  \leq  \min\{\nu_{\tilde{p}},\bar{\gamma}_{\tilde{x}}^{-1}(\nu_{\tilde{x}})\},
\label{eq-proof-thm-dyn-nu-1}
\end{equation}
where $\bar{\gamma}_{\tilde{x}}$ is a class $\mathcal{K}_{\infty}$ function generated by Lemma \ref{lem:bound_tilde_x}. Let $\underline{\chi}$ and $\bar{\chi}$ be the class $\mathcal{K}_{\infty}$ functions generated by Lemma \ref{lem:bound_mu} (which are independent of the choice of the parameters $p_{i}$, $i\in\{1,\ldots,N\}$, see Lemma \ref{lem:bound_mu}). We introduce $\Delta_{\infty}\in(0,\Delta)$ which is sufficiently small such that
\begin{equation}
\begin{array}{rllll}
\underline{\chi}^{-1}\circ\bar{\chi}(\pi(\Delta_{\infty},0))  <  \nu_{1}  & \text{and} &
\Delta_{\infty} < \nu_{1},
\end{array}\label{eq-proof-thm-dyn-Delta-infty}
\end{equation}
where $\pi$ is the class $\mathcal{KL}$ function in (\ref{eq:dyn-sampling}).  It is always possible to ensure (\ref{eq-proof-thm-dyn-Delta-infty}) as $\underline{\chi},\,\bar{\chi}\in\mathcal{K}_{\infty}$ and $\pi\in\mathcal{KL}$. We select $T>0$ sufficiently large such that the conclusion of Lemma \ref{lem:bound_mu} holds with $\epsilon>0$ sufficiently small such that
\begin{equation}
\begin{array}{rllll}
\underline{\chi}^{-1}\circ\bar{\chi}(\pi(\Delta_{\infty},0)+\epsilon)  <  \nu_{1}  & \text{and} &
\underline{\chi}^{-1}(\epsilon) < \alpha \Delta_{\infty} < \nu_{1}.
\end{array}\label{eq-proof-thm-dyn-epsilon}
\end{equation}
We take $N^{\star}\in\mathbb{N}_{\geq 1}$ sufficiently large such that
\begin{equation}
\begin{array}{llllllll}
\underline{\chi}^{-1}\left(\bar{\chi}(\pi(s,N^{\star}))+\epsilon\right) & \leq & \alpha s & & \forall s\in[\Delta_{\infty},\Delta],
\end{array}\label{eq-proof-thm-dyn-N-star}
\end{equation}
which is always possible as $\underline{\chi},\,\bar{\chi}\in\mathcal{K}_{\infty}$, $\pi\in\mathcal{KL}$ and $\underline{\chi}^{-1}(\epsilon)<\alpha\Delta_{\infty}\leq \alpha s$ for $s\in [\Delta_{\infty},\Delta]$ according to (\ref{eq-proof-thm-dyn-epsilon}). A sufficient condition to ensure (\ref{eq-proof-thm-dyn-N-star}) is $\underline{\chi}\left(\bar{\chi}(\pi(\Delta,N^{\star}))+\epsilon\right) \leq \alpha\Delta_{\infty}$ for example.

Let $N\geq N^{\star}$, $T^{\star} = T$, $T_{d}\geq T^{\star}$, $(x(0),\tilde{x}_{i}(0)) \in \mathcal{H}(0,{\Delta_x}) \times \mathcal{H}(0,{\Delta_{\tilde{x}}})$ for $i\in\{1,\dots,N\}$ and $u \in \mathcal{M}_{\Delta_u}$ such that Assumption \ref{ass:PE_y_tilde} holds. We proceed by induction. Like in (\ref{eq:tildep_f-prel}), using Lemma \ref{lem:bound_mu}, we derive that, since $t_{1}=T_{d}\geq T$ in view of (\ref{eq:update_time_interval}) and since $p^{\star}\in\Theta(0)=\Theta$, 
\begin{equation}
\begin{array}{llllll}
|\tilde{p}_{\sigma(t_{1}^{-})}(t_{1}^{-})|_{\infty} & \leq & \underline{\chi}^{-1}\left(\bar{\chi}(d(p^{\star},\widehat{\Theta}(0))+\epsilon)\right)\\
& \leq & \underline{\chi}^{-1}\left(\bar{\chi}(\pi(\Delta(0),N)+\epsilon)\right).
\end{array}\label{eq-proof-thm-dyn-tilde-p-t1}
\end{equation}
We know that $\Delta_{\infty}<\Delta(0)=\Delta$, hence, in view of (\ref{eq-proof-thm-dyn-N-star})
\begin{equation}
\begin{array}{llllll}
|\tilde{p}_{\sigma(t_{1}^{-})}(t_{1}^{-})|_{\infty} & \leq & \alpha\Delta(0) & = & \Delta(1).
\end{array}\label{eq-proof-thm-dyn-t1}
\end{equation}
We deduce from (\ref{eq-proof-thm-dyn-t1}) that $p^{\star}\in\mathcal{H}(\hat{p}(t_{1}^{-}),\Delta(1))$. On the other hand, $p^{\star}\in\Theta(0)$ which implies that $p^{\star}\in \Theta(1)=\mathcal{H}(\hat{p}(t_{1}^{-}),\Delta(1))\cap\Theta(0)$.

Let $k\in\mathbb{N}_{\geq 1}$. Since $t_{k+1}-t_{k}=T_{d}\geq T$ in view of (\ref{eq:update_time_interval}), we obtain as in (\ref{eq-proof-thm-dyn-tilde-p-t1})
\begin{equation}
\begin{array}{llllll}
|\tilde{p}_{\sigma(t_{k}^{-})}(t_{k}^{-})|_{\infty} & \leq & \underline{\chi}^{-1}\left(\bar{\chi}(d(p^{\star},\widehat{\Theta}(k-1))+\epsilon)\right)\\
& \leq & \underline{\chi}^{-1}\left(\bar{\chi}(\pi(\Delta(k-1),N)+\epsilon)\right).
\end{array}
\end{equation}
If $\Delta(k-1)\leq \Delta_{\infty}$, then, in view of (\ref{eq-proof-thm-dyn-Delta-infty}),
\begin{equation}
\begin{array}{lllllllll}
|\tilde{p}_{\sigma(t_{k}^{-})}(t_{k}^{-})|_{\infty} & \leq & \Delta_{\infty} & \leq & \nu_{1} & \leq & \nu_{\tilde{p}}.
\end{array}
\end{equation}
If $\Delta(k-1)>\Delta_{\infty}$, then, in view of (\ref{eq-proof-thm-dyn-N-star})
\begin{equation}
\begin{array}{lllllllll}
|\tilde{p}_{\sigma(t_{k}^{-})}(t_{k}^{-})|_{\infty} & \leq & \alpha\Delta(k-1) = \Delta(k).
\end{array}
\end{equation}
In this case $p^{\star}\in\mathcal{H}(\hat{p}(t_{k}^{-}),\Delta(k))$ and we know that $p^{\star}\in \Theta(0)\cap\ldots\cap\Theta(k-1)$. As a consequence, $p^{\star}\in\Theta(k)$.

We have shown that for any $k\in\mathbb{N}_{\geq 1}$
\begin{equation}
\begin{array}{lllllllll}
|\tilde{p}_{\sigma(t_{k}^{-})}(t_{k}^{-})|_{\infty} & \leq & \left\{\begin{array}{lllllll}
\nu_{1} & \text{when} & \Delta(k-1)\leq \Delta_{\infty}\\
\Delta(k) & \text{when} & \Delta(k-1) > \Delta_{\infty}.
\end{array}\right.
\end{array}
\end{equation}
Since $k\to\infty$ in view of (\ref{eq:update_time_interval}) and $\Delta(k)=\alpha^{k}\Delta\to 0$ as $k\to\infty$, we deduce that there exists $\overline{T}>0$ such that
\begin{equation}
\begin{array}{llllllll}
|\tilde{p}_{\sigma(t)}(t)| & \leq & \nu_{1} & \leq & \nu_{\tilde{p}} & & & \forall t\geq \overline{T}.
\end{array}
\end{equation}
The proof is completed by following the same lines as in the proof of Theorem \ref{thm:main_static}. \hfill $\blacksquare$

\subsection{Proof of Proposition \ref{prop:lin_obs_error}} \label{app:lin_obs_error}
For any $p_i\in\Theta$, $i\in\{1,\dots,N\}$, we obtain the following state estimation error system
\begin{eqnarray}
	\dot{\tilde{x}}_{i} & = & \Big( A(p_i) + L(p_{i}) C(p_i) \Big) \hat{x}_{i} - \Big( A(p^{\star}) + L(p_{i}) C(p^{\star}) \Big) x + \Big( B(p_i) - B(p^{\star}) \Big) u \nonumber \\
	& = & \Big( A(p_i) + L(p_{i}) C(p_i) \Big) \tilde{x}_{i} + \Big( \tilde{A}(p_i,p^{\star}) + L(p_{i}) \tilde{C}(p_i,p^{\star}) \Big) x + \tilde{B}(p_i,p^{\star}) u, \label{eq:lin_obs_error}
\end{eqnarray}
where we denote $\tilde{A}(p_i,p^{\star}):=A(p_i)-A(p^{\star})$, $\tilde{B}(p_i,p^{\star}):=B(p_i)-B(p^{\star})$ and $\tilde{C}(p_i,p^{\star}):=C(p_i)-C(p^{\star})$.
Let $p_{i}\in\Theta$ and $V_{i}:\tilde{x}_{i}\mapsto\tilde{x}_{i}^{T}P_{i} \tilde{x}_{i}$, where $P_i$ is a real symmetric, positive definite matrix which satisfies
\begin{equation}
	P_i\big( A(p_i) + L(p_{i}) C(p_i)\big) + \big( A(p_i) + L(p_{i}) C(p_i)\big)^{T}P_i = -\nu_{i} \mathbb{I}, \textrm{with } \nu_{i}> 1. \label{eq:lyap_eq_lin}
\end{equation}
Such a matrix $P_{i}$ always exists according to Theorem 3.6 of \cite{khalil1996nonlinear2} since $A(p_i) + L(p_{i}) C(p_i)$ is Hurwitz. Hence, \eqref{eq:lyap_iss_1} is satisfied with  $a_1:= \min_{i\in\{1,\dots,N\}} \lambda_{\min}(P_{i})$ and $a_2:=n_x \max_{i\in\{1,\dots,N\}} \lambda_{\max}(P_{i})$. Let $\tilde{x}_{i},x\in\mathbb{R}^{n_{x}}$ and $u\in\mathbb{R}^{n_{u}}$, it holds that
{\setlength\arraycolsep{1pt}\begin{equation}
\begin{array}{llllll}
	\left\langle \nabla V_{i}(\tilde{x}_{i}),F_{i}(\tilde{x}_{i}, x, \tilde{p}_{i}, p^{\star}, u)  \right\rangle & = & \tilde{x}_{i}^{T} \Big( P_i\big( A(p_i) + L(p_{i}) C(p_i)\big) + \big( A(p_i) + L(p_{i}) C(p_i)\big)^{T}P_i \Big)  \tilde{x}_{i} \\
	&& \hspace{-0.7cm}+ 2 \tilde{x}_{i}^{T} P_i \Big( \big( \tilde{A}(\tilde{p}_{i}+p^{\star},p^{\star}) + L(p_{i}) \tilde{C}(\tilde{p}_{i}+p^{\star},p^{\star}) \big) x + \tilde{B}(\tilde{p}_{i}+p^{\star},p^{\star}) u \Big)  \\
	& \leq & - \nu_{i} |\tilde{x}_{i}|^{2} + 2 \left|\tilde{x}_{i}\right| |P_i| \gamma(\tilde{p}_{i},x,u), \label{eq:lin_V_i_1}
\end{array}
\end{equation}}
\hspace{-0.15cm}where $\gamma(\tilde{p}_{i},x,u):=\underset{p\in\Theta}{\max}\left|\Big( \big( \tilde{A}(\tilde{p}_{i}+p,p) +  L(p_{i}) \tilde{C}(\tilde{p}_{i}+p,p) \big) x +  \tilde{B}(\tilde{p}_{i}+p,p^{\star}) u \Big) \right|$. The function $\gamma$ is continuous since $A$, $B$ and $C$ are continuous in their argument and $\Theta$ is compact set. Moreover $\gamma(0,x,u)=0$ for any $x\in\mathbb{R}^{n_{x}}$ and $u\in\mathbb{R}^{n_{u}}$. Using the fact that $2ab\leq \frac{\nu_{i}}{2}a^{2} + \frac{2}{\nu_{i}}b^{2}$ for any $a,b\in\mathbb{R}$, we deduce from (\ref{eq:lin_V_i_1}) that, for any
$\tilde{x}_{i},x\in\mathbb{R}^{n_{x}}$ and $u\in\mathbb{R}^{n_{u}}$,
\begin{equation}
\begin{array}{llllll}
	\left\langle \nabla V_{i}(\tilde{x}_{i}),F_{i}(\tilde{x}_{i}, x, \tilde{p}_{i}, p^{\star}, u)  \right\rangle
	& \leq & - \frac{\nu_{i}}{2} |\tilde{x}_{i}|^{2} + \frac{2}{\nu_{i}}|P_i|^{2} \gamma(\tilde{p}_{i},x,u)^{2} \label{eq:lin_V_i_2}
\end{array}
\end{equation}
from which we derive that \eqref{eq:lyap_iss_2} is satisfied by using (\ref{eq:lyap_iss_1}) and by invoking the equivalence of the infinity and the Euclidean norms. \hfill $\Box$

\subsection{Proof of Proposition \ref{prop:cc_obs_error}} \label{app:cc_obs_error}
Let $p_{i}\in\Theta$ with $i\in\{1,\dots,N\}$, we obtain the following state estimation error system from \eqref{eq:non_lin_plant} and \eqref{eq:non_lin_obs}
\begin{eqnarray}
	\dot{\tilde{x}}_{i} & = & \Big(A(p_i) + L(p_{i}) C(p_i) \Big)\tilde{x}_{i} + G(p_i) \Big( \gamma(w_i) - \gamma(v) \Big) \nonumber \\
	&& + \tilde{G}(p_i,p^{\star}) \gamma(v) + \Big(\tilde{A}(p_i,p^{\star})+ L(p_{i}) \tilde{C}(p_i,p^{\star})\Big)x + \tilde{B}(p_i,p^{\star}) \phi(u,y), \label{eq:tilde_x_1}	 				
\end{eqnarray}
where $v:=Hx$, $w_{i}:=H\hat{x}_{i} + K(p_{i})(C(p_i)\hat{x}_{i}-y)$, $\tilde{A}(p_i,p^{\star}):=A(p_i)-A(p^{\star})$, $\tilde{B}(p_i,p^{\star}):=B(p_i)-B(p^{\star})$, $\tilde{G}(p_i,p^{\star}):=G(p_i)-G(p^{\star})$ and $\tilde{C}(p_i,p^{\star}):=C(p_i)-C(p^{\star})$.
In view of (\ref{ass:nonlinear_functions}) and according to the mean value theorem, there exists $\delta(t)=\mathrm{diag}(\delta_{1}(t),\dots,\delta_{n_{\gamma}}(t))$, where $\delta_{k}(t)$ take values in the interval $[a_{\gamma_k},b_{\gamma_k}]$ so that, for $\gamma = (\gamma_{1},\dots,\gamma_{n_\gamma})$,
\begin{equation}
	\gamma(w_i) - \gamma(v) = \delta(t) (w_i - v), \qquad \forall w_i, v\in \mathbb{R}^{n_\gamma},\,i\in\{1,\dots,N\}. \label{eq:gamma_delta}
\end{equation}
We define $V_{i}:\tilde{x}_{i}\mapsto\tilde{x}_{i}^{T}P_{i}\tilde{x}_{i}$, where $P_{i}$ is a real symmetric, positive definite matrix given by \eqref{eq:lmi_cc}. Note that $V_i$ satisfies inequality \eqref{eq:lyap_iss_1} of Assumption \ref{ass:obs_error_i} with $a_{1}=\lambda_{\min}(P_{i})$ and $a_{2}=n_{x}\lambda_{\max}(P_{i})$. By following the proof of Theorem 2 in \cite{chong2012robust} with the vector $\chi_{i}:=(\tilde{x}_{i},\delta(t)(H+K(p_{i})C)\tilde{x}_{i}, \bar{w})$, where $\bar{w}=\bar{w}(p_i,p^{\star},x,u) :=  \tilde{G}(p_i,p^{\star}) \gamma(Hx) + \big( \tilde{A}(p_i,p^{\star}) + L(p_{i}) \tilde{C}(p_i,p^{\star}) \big)x + \tilde{B}(p_i,p^{\star}) \phi(u,y)$, we obtain, for any $\tilde{x}_{i},x\in\mathbb{R}^{n_{x}}$, $u\in\mathbb{R}^{n_{u}}$,
\begin{eqnarray}
	\left\langle\nabla V_{i}(\tilde{x}_{i}),F_{i}(\tilde{x}_{i}, x, \tilde{p}_{i}, p^{\star}, u)\right\rangle & \leq & -\nu_{i}|\tilde{x}_{i}|^{2} + \mu_{i} |\bar{w}|^{2}. \label{eq:V_3}
\end{eqnarray}
We then use the same arguments as in the proof of Proposition \ref{prop:lin_obs_error} to derive (\ref{eq:lyap_iss_2}). \hfill $\Box$


\bibliographystyle{plain}
\bibliography{../multiobs_ref}

\end{document}